\documentclass{article}
\title{\vspace{-50pt}
	\Huge \text{Representation Varieties of Stacks}\\ 
\text{and Trace Maps}}
\author{Jacob Erlikhman}
\date{}
\usepackage[T1]{fontenc}

\usepackage{amsthm,amsmath,amssymb,enumitem,comment,tikz-cd,hyperref,tikz,braket,latexsym,mathtools}
\usepackage[backend=biber,style=alphabetic]{biblatex}
\DeclareSymbolFont{yhlargesymbols}{OMX}{yhex}{m}{n}
\DeclareMathAccent{\ywidetilde}{\mathord}{yhlargesymbols}{"65}
\usepackage[scaled=.92]{helvet}
\usetikzlibrary{calc,patterns,angles}
\hypersetup{hidelinks}
\usepackage{xifthen}
\usepackage{fontawesome}
\usepackage[scr=boondoxo,bb=boondox]{mathalpha}
\usepackage{import}
\usepackage{pdfpages}
\usepackage{transparent}
\usepackage{bm}

\numberwithin{equation}{section}

\makeatletter
\newcommand\RedeclareMathOperator{%
  \@ifstar{\def\rmo@s{m}\rmo@redeclare}{\def\rmo@s{o}\rmo@redeclare}%
}
\newcommand\rmo@redeclare[2]{%
  \begingroup \escapechar\m@ne\xdef\@gtempa{{\string#1}}\endgroup
  \expandafter\@ifundefined\@gtempa
     {\@latex@error{\noexpand#1undefined}\@ehc}%
     \relax
  \expandafter\rmo@declmathop\rmo@s{#1}{#2}}
\newcommand\rmo@declmathop[3]{%
  \DeclareRobustCommand{#2}{\qopname\newmcodes@#1{#3}}%
}
\@onlypreamble\RedeclareMathOperator
\makeatother

\usepackage[top=2.4cm, bottom=2.4cm, left=2.4cm, right=2.4cm]{geometry}
\newtheoremstyle{quest}{\topsep}{\topsep}{}{}{\bfseries}{}{ }{\thmname{#1}\thmnote{ #3}.}
\theoremstyle{quest}

\theoremstyle{definition}

\newtheorem{thm}{Theorem}[section]
\newtheorem{lem}[thm]{Lemma}

\newtheorem{prop}[thm]{Proposition}
\newtheorem{ex}[thm]{Example}

\newtheorem{cor}[thm]{Corollary}

\theoremstyle{remark}
\newtheorem{rmk}[thm]{Remark}

\newcommand{\opnorm}[1]{{\left\vert\kern-0.25ex\left\vert\kern-0.25ex\left\vert #1 
		\right\vert\kern-0.25ex\right\vert\kern-0.25ex\right\vert}}

\newcommand{\R}{\mathbb{R}}

\newcommand{\msf}{\mathsf}

\newcommand{\Mod}{-\msf{mod}}

\DeclareMathOperator{\GL}{GL}

\DeclareMathOperator{\Hom}{Hom}
\DeclareMathOperator{\Spec}{Spec}

\DeclareMathOperator{\tr}{tr}

\DeclareMathOperator{\SL}{SL}

\DeclareMathOperator{\id}{id}

\DeclareMathOperator{\End}{End}

\DeclareMathOperator{\colim}{colim}

\DeclareMathOperator{\Map}{Map}

\DeclareMathOperator{\rk}{rk}
\RedeclareMathOperator{\O}{O}

\DeclareMathOperator{\pt}{pt}

\DeclareMathOperator{\Sym}{Sym}

\DeclareMathOperator{\Rep}{Rep}

\RedeclareMathOperator*{\colim}{colim}

\DeclareMathOperator{\Coh}{Coh}

\DeclareMathOperator{\Quot}{Quot}

\RedeclareMathOperator{\Mod}{\text{-}\mathsf{mod}}
\DeclareMathOperator{\coker}{coker}

\RedeclareMathOperator{\Spec}{Spec}
\RedeclareMathOperator{\Map}{Map}

\DeclareMathOperator{\Perf}{Perf}
\DeclareMathOperator{\APerf}{APerf}

\RedeclareMathOperator{\Bar}{Bar}
\RedeclareMathOperator{\mod}{-mod}

\DeclareMathOperator{\Vect}{Vect}

\DeclareMathOperator{\triv}{triv}

\newcommand{\git}{\mathbin{
  \mathchoice{/\mkern-6mu/}
    {/\mkern-6mu/}
    {/\mkern-5mu/}
    {/\mkern-5mu/}}}
\begin{document}
\maketitle
\begin{abstract}
	We define a derived stack $\mathscr{Rep}_n(X)$ which generalizes the assignment $A\leadsto \Rep_n(A)$ to an algebra of its derived $\GL_n$-representation variety of \cite{BKR} from algebras $A$ to perfect stacks $X$ over characteristic 0 fields. In the case of a quasi-projective classical scheme $X$, we show that $\mathscr{Rep}_n(X)$ admits a subfunctor $\Quot^{n,\text{fr}}_{\mathscr{O}_X^n}(X)\subset \mathscr{Rep}_n(X)$, which is in fact represented by a derived scheme almost of finite type. We further construct a Fourier-Mukai integral transform between the derived categories of quasi-coherent sheaves on $X$ and $\Quot^{n,\text{fr}}_{\mathscr{O}_X^n}(X)$ which induces a trace map at the level of Hochschild homology generalizing the trace morphism constructed in \cite{BKR} (for finitely presented commutative algebras). We show that the subfunctor $\Quot^{n,\text{fr}}_{\mathscr{O}_X^n}(X)$ is a derived enhancement of the framed locus of the Quot scheme of points and that this derived scheme is a $\GL_n$-torsor over the stack of coherent length $n$ torsion sheaves. Hence, this stack is an analog of the derived character stack for quasi-projective schemes, and we show that for smooth, Calabi-Yau $X$, it inherits a shifted symplectic structure in the sense of \cite{PTVV} from the one constructed in \cite{BD} on the moduli stack of perfect complexes with proper support. This structure is shown to give a generalization of the classical symplectic structure on character varieties of surfaces of genus $ 1$ constructed by Goldman \cite{gold}.
\end{abstract}
 
\tableofcontents
\section{Introduction}
\subsection{Motivation}
Given a topological surface $\Sigma$, its ($\SL_2$-)representation and character varieties have been of classical interest since the beginning of the $20^{\text{th}}$ century (see the introduction to \cite{goldman-char-variety-history}). Classically, one fixes a (connected, reductive) Lie group $G$, and studies either representations of the fundamental group
\begin{align*}
	\Hom_{\mathsf{Gp}}(\pi_1(\Sigma),G)
\end{align*} 
or these up to conjugacy
\begin{align*}
	\Hom_{\mathsf{Gp}}(\pi_1(\Sigma),G) /G,
\end{align*} 
where $G$ acts by the conjugation action. In the case of $G=\GL_n$, one can give representation varieties the structure of an affine algebraic variety by defining a representable functor \cite[Proposition 1.2]{lub}
\begin{align*}
	\Rep_n(\pi_1(\Sigma)): \mathsf{Aff}^{\text{op}}\to &\mathsf{Set}\\
	\Spec B\mapsto &\Hom_{\mathsf{Gp}}(\pi_1(\Sigma),\GL_n(B)),
\end{align*} 
where $\mathsf{Aff}^{\text{op}}$ denotes the opposite of the category of affine $\mathbb{C} $-schemes. By the adjunction
\[
\begin{tikzcd}
\mathsf{Gp}\arrow[r,shift left,"{\mathbb{C}[-]} "]&\mathsf{Alg}_{\mathbb{C} }\arrow[l,shift left, "(-)^\times"]
\end{tikzcd}
\] 
and the observation that $M_n(\mathbb{C} )^\times=\GL_n(\mathbb{C} )$, where $M_n(\mathbb{C} )$ is the $\mathbb{C} $-algebra of $(n\times n)$-matrices with complex entries, one can reformulate this functor as
\begin{align*}
	\Rep_n(\pi_1(\Sigma))=\Rep_n(A)^{\text{cl}}:\Spec B\mapsto \Hom_{\mathsf{Alg}_{\mathbb{C} }}(A,M_n(B )),
\end{align*} 
where $M_n(B)$ denotes $(n\times n)$-matrices with entries in the $\mathbb{C} $-algebra $B$ and $A\coloneq \mathbb{C} [\pi_1(\Sigma)]$. Similarly, the $\GL_n$ \textbf{character stack} is defined as
\begin{align*}
	\Rep_n(\pi_1(\Sigma)) /\GL_n = \Rep_n(A) ^{\text{cl}}/\GL_n.
\end{align*} 
On the other hand, the associated $\GL_n$ \textbf{character variety} is the GIT quotient $\Rep_n(A)^{\text{cl}}\git \GL_n$. In the algebraic setting, the derived enhancements of these objects, denoted $\Rep_n(A) $ and $\Rep_n(A) \git \GL_n$, respectively, are already known in the literature, first constructed in \cite{BKR}. In this paper, we study the derived and stacky versions of these constructions, so that there is a correspondence
\begin{align*}
	\Rep_n(A) /\GL_n\leftarrow  (\Rep_n(A) /\GL_n)^{\text{cl}}\to \Rep_n(A)^{\text{cl}}\git \GL_n\to \Rep_n(A) \git \GL_n.
\end{align*} 
In the 1980s, Goldman defined in \cite{gold} a symplectic structure on
\begin{align*}
	\Hom(\pi_1(\Sigma),G) \git G,
\end{align*} 
and its quantization has been of interest in knot theory; see e.g. \cite{turaev-skein,AMR96,AMR98} for the classical story and \cite{BZBJ,jordan-survey} for a modern formulation.

When $A$ is finitely presented and commutative (so that $\Spec A$ is a quasi-projective scheme), the functor
\begin{align*}
	B\mapsto \Hom_{\mathsf{Alg}_{\mathbb{C} }}(A,M_n(B))
\end{align*} 
is also the open subfunctor of the functor of points of the Quot scheme parametrizing length $n$ quotients 
\begin{align*}
	q:(A\otimes_{\mathbb{C} }B)^n\to M
\end{align*} 
for which the induced map $B^n\to M$ is an isomorphism. We call this the framed locus and explain this in detail in Section \ref{sec:coherent-and-quot-reps}. Like the specific case of Hilbert schemes, Quot schemes are a basic example of moduli problems in algebraic geometry and thus have been of interest since Grothendieck's original proof of their representability \cite{grothendieck-quot,nitsure-quot,huybrechts-lehn}. They are used as covers of the stack of coherent sheaves \cite{huybrechts-lehn}; thus, they are also important in the theory of cohomological Hall algebras \cite{KV,PS}.

\subsection{Overview}
This paper generalizes results of \cite{BKR} and \cite{symplectic-structure} in several directions. We begin by placing a trace map of \cite{BKR} in a trace formalism context \cite{TFT,HSS,nonlinear-trace}. Next, we generalize the study of such representation and character varieties  for $G=\GL_n$ from $A\leadsto \Rep_n(A)$ to $X\leadsto \Rep_n(X)$, where $X$ is now any perfect stack (see \cite{BFN} for what this means). In fact, our construction takes as input a category $\mathcal{C}$ of the form $\mathsf{QCoh}(X)$ or $A\Mod$ equipped with the natural functor to $k\Mod$ given by pushforward/restriction; in the latter case $\mathcal{C}=A\Mod$, the choice of this restriction functor is equivalent to choosing a perfect generator, so this construction is not Morita invariant. In contrast, the closely related stack $\mathscr{M}_{\mathcal{C}}$ of \cite{TV} \textit{is} Morita invariant. In the process, we construct a derived enhancement of an open subfunctor of the Quot scheme of $n$ points which parametrizes quotients $q:\mathscr{O}^n_{X\times S}\twoheadrightarrow \mathscr{E}$, where $\mathscr{E}$ is an $S$-flat family of length $n$ torsion sheaves and the adjoint map to $q$, $\mathscr{O}^n_S\to p_{S*}\mathscr{E}$, is an isomorphism for $p_S:X\times S\to S$. Such derived Quot schemes have been constructed before, for projective $X$ in \cite{BKSv2} and for quasi-projective $X$ in \cite{quot} (see also \cite{kapranov-quot}); however, our construction differs significantly from theirs, as it uses a fiber of the stack of coherent sheaves as input instead of constructing the functor of points directly (see Section \ref{sec:coherent-and-quot-reps}, in particular, Corollary \ref{cor:quot-is-a-scheme}, for details). On their common domains of definition, we expect all three constructions to agree. We check this explicitly for the construction in \cite{quot} for smooth quasi-projective schemes in the appendix.

We show that in the case of a quasi-projective scheme $X$, there is a Fourier-Mukai integral transform between quasi-coherent sheaves on this derived Quot scheme and $\mathsf{QCoh}(X)$. This functor induces a trace (in the sense of \cite{HSS,nonlinear-trace}) map at the level of Hochschild homology which generalizes the trace map of \cite{BKR} from finitely presented commutative algebras to arbitrary quasi-projective schemes, and gives it an especially natural characterization: It arises due to the existence of a universal almost perfect complex on (the noncommutative analog of) $\Spec A \times \Rep_n(A)$; see Sections \ref{sec:trace-maps-I} and \ref{sec:universal-sheaves}.

We next construct a generalization (in the sense of \cite{PTVV}) of the Goldman symplectic structure to (derived) character stacks equipped with a Calabi-Yau structure whose underlying classical truncation is of the form
\begin{align*}
	\Hom(A,M_n) /\GL_n,
\end{align*} 
where $A$ is a commutative smooth algebra. In \cite{PTVV}, the authors construct a shifted symplectic structure on the derived moduli stack of maps $\Map(X_B,BG)$, where for a (compact, oriented) topological manifold $X$, $X_B$ is the constant derived stack with value the singular complex of $X$. This moduli stack is the derived character stack of $X$; moreover, in the case that $X$ is a compact oriented 2-manifold,  the shifted symplectic structure becomes 0-shifted; see \cite[Theorem 2.5]{PTVV} and the examples following it. Further, this canonical symplectic structure arises from the choice of a symmetric bilinear form on the Lie algebra $\mathfrak g$ of $G$. This symplectic structure on the derived character stack of $X$ recovers the Goldman symplectic structure of \cite{gold} by pulling back to the classical character stack and then restricting to the appropriate smooth, classical locus \cite{PTVV,safronov-goldman}. In Section 2.3 of \cite{PTVV}, the authors also construct a symplectic structure on the moduli stack of perfect complexes on a smooth and proper Calabi-Yau scheme. It turns out that the properness assumption can be dropped \cite{BD}, \cite[Theorem D and Lemma 3.7]{BD-is-PTVV}. We use these foundations to construct our symplectic structure by pulling back this already existing one. Symplectic structures on derived character varieties have been constructed before in \cite{symplectic-structure}. Therein, a generalization of derived representation varieties to categories is considered, where algebras are viewed as categories with one object. We did not take this approach, instead viewing algebras as categories (noncommutative spaces) via their categories of modules. Our constructions are functorial with respect to that view (as in e.g. \cite{TV}). We now give a more detailed overview of the results below.
\subsection{Contents}
Let $k$ be a characteristic 0 field, fixed for the rest of this introduction (in fact, for the rest of this paper). In \cite{BKR}, the authors construct and study a trace map 
\begin{align*}
	\mathrm H\mathrm C_*(A)\to \mathbb{L}(A)_n
\end{align*} 
between the cyclic homology of an associative $k$-algebra $A$ and its ``representation homology,'' $\mathbb{L}(A)_n$. This latter homology is by definition derived global functions on the derived $n$-dimensional representation variety of $A$, defined as follows (see Section \ref{sec:notation-and-conventions} for context). As earlier in this introduction, let $\Rep_n(A)^{\text{cl}}$ be the affine variety which represents the functor
\begin{align*}
	B\mapsto \Hom_{\mathsf{Alg}_k}(A,M_n(B)),
\end{align*} 
where $M_n:\mathsf{CAlg}_k\to \mathsf{Alg}_k$ takes a commutative $k$-algebra $B$ to the algebra of $(n\times n)$-matrices with entries in $B$. This scheme is called the ($n$-dimensional) \textbf{representation variety} of $A$. In particular, the functor $A\mapsto \mathscr{O}(\Rep_n(A)^{\text{cl}})$ is left adjoint to the functor $M_n$. All these constructions can now be derived as in \cite{BKR}, to obtain a derived representation variety and a left derived functor $\mathbb{L}(-)_n$, and this latter functor is called the \textbf{representation homology}. When $A=k[\pi_1(\Sigma)]$ for a surface $\Sigma$, this is a derived enhancement of the cohomology of the classical representation variety of $A$.

In Section \ref{sec:domain-walls}, we place the above into a trace formalism context (see \cite{HSS,nonlinear-trace} for the trace formalism we use). Given algebras $A$, $\mathbb{L}(A)_n$, we show that the categories of modules over these algebras admit a  natural ``domain wall,'' i.e., a Fourier-Mukai integral transform with a continuous right adjoint (explained in Sections \ref{sec:domain-walls} and \ref{sec:universal-sheaves}), between them. This point of view gives a natural reason for why the trace map of \cite{BKR} should exist: It follows from the existence of a well-behaved integral kernel (Lemma \ref{lem:A-module-structure} and Proposition \ref{prop:R-is-a-domain-wall}). Indeed, by \cite[Section 3.4]{nonlinear-trace}, this integral transform gives rise to a map
\begin{align*}
	\mathrm H\mathrm H_*(A)\to \mathrm H\mathrm H_*(\mathbb{L}(A)_n).
\end{align*} 
Because $\mathbb{L}(A)_n$ is a commutative algebra, it admits a map
\begin{align*}
	\mathrm H\mathrm H_*(\mathbb{L}(A)_n)\to \mathbb{L}(A)_n.
\end{align*} 
Moreover, the composition
\begin{align*}
	\mathrm H\mathrm H_*(A)\to \mathbb{L}(A)_n
\end{align*} 
is equivariant with respect to the natural $S^1$-action on $\mathrm H\mathrm H_*(A)$ and the trivial $S^1$-action on $\mathbb{L}(A)_n$ \cite[Theorem 2.14]{HSS}. Since $S^1$-coinvariants form an adjunction \cite[Section I.1]{NS}
\begin{align*}
	(-)_{S^1}=\colim_{BS^1}:\mathsf{Vect}^{S^1}\to  \mathsf{Vect}:\triv,
\end{align*} 
where $\triv$ equips a vector space with the trivial $S^1$-action and $\mathsf{Vect}^{S^1}\coloneqq \mathsf{Map}(BS^1,\mathsf{Vect})$ denotes vector spaces with $S^1$-actions, we obtain a map
\begin{align*}
	\mathrm{HC}_*(A)\to \mathbb{L}(A)_n.
\end{align*} 
\begin{thm}[Theorem \ref{thm:trace} in the text]\label{thm:intro-trace-map}
	If $A$ is a connective associative algebra, then the induced map
\begin{align*}
	\mathrm H\mathrm C_*(A)\to \mathbb{L}(A)_n
\end{align*} 
is the trace map of \cite{BKR}.
\end{thm}

In Section \ref{sec:ext-to-stacks}, we extend this definition of derived representation varieties from algebras $A$ to perfect stacks $X$ (see \cite{BFN} for what this means). Let $\mathcal{X}_{A\Mod}$ denote the stack which represents the functor $\Spec B\mapsto \ywidetilde{(A\otimes B)\Mod}$, where $\ywidetilde{\mathcal{C}}$ denotes the maximal subgroupoid of a category $\mathcal{C}$. Similarly, let $\mathcal{X}_{\mathsf{QCoh}(X)}$ for a perfect stack $X$ denote the stack of quasi-coherent complexes on $X$, defined in Section \ref{sec:ext-to-stacks}. We equip both of these stacks with natural maps to the stack of quasi-coherent complexes on a point, $\Vect\coloneqq k\text{-mod}$, defined by pushforward (Lemma \ref{lem:perf-pushforward}). Using either of these categories together with their maps to $\Vect$, we define a derived stack
\begin{align*}
	\mathscr{Rep}_n(\mathcal{C})\coloneqq \mathcal{X}_{\mathcal{C}}\times _{\Vect}\{k^n\} ,
\end{align*} 
where the map $\{k^n\} \to \Vect$ classifies $k^n$. Because the map $ \mathcal{X}_{\mathcal{C}}\to \Vect$ is part of the construction of this moduli stack, this construction is \textit{not} Morita invariant. We show (Theorem \ref{thm:rep-on-affines}) that the extension, which we call the ``$n$-dimensional representation moduli,'' agrees with representation varieties in the case of $\mathcal{C}=A\Mod$. We compute this stack in the case that $\mathcal{C}=\mathsf{QCoh}(X)$ for $X=\mathbb{P} ^1$ and $X=BG$ with $G$ a reductive algebraic group.

Further, we show in Section \ref{sec:trace-maps-II} that for $\mathcal{C}=\mathsf{QCoh}(X)$ with $X$ a classical quasi-projective scheme, this stack admits a substack of ``coherent representations,'' denoted $\mathscr{Rep}_n^{\text{coh}}(X)$. In Theorem \ref{thm:coherent-reps}, we prove that this substack is 1-Artin, locally almost of finite type, and with affine diagonal. Finally, we show in Corollary \ref{cor:quot-is-a-scheme} that the stack of coherent representations admits a substack 
\begin{align*}
	\Quot^{n,\text{fr}}_{\mathscr{O}^n_X}(X)\subset \mathscr{Rep}_n^{\text{coh}}(X)\subset \mathscr{Rep}_n(\mathsf{QCoh}(X))
\end{align*} 
which is a derived scheme almost of finite type and with affine diagonal whose classical truncation is the open subfunctor of the Quot scheme of points which parametrizes quotients $q: \mathscr{O}^n_{X\times S}\to \mathscr{E}$ for which the induced map $\mathscr{O}^n_S\to p_{S*}\mathscr{E}$ is an isomorphism.  On affine schemes of finite presentation, all of these stacks agree.
\begin{thm}[Proposition \ref{prop:quot-scheme-is-reps} and Corollary \ref{cor:coherent-reps-match-reps} in the text]\label{thm:intro-affine-reps-all-agree}
Let $X=\Spec A$ be an affine scheme of finite presentation. Then
\begin{align*}
	\Quot^{n,\text{fr}}_{\mathscr{O}_X^n}(X)\cong \mathscr{Rep}_n^{\text{coh}}(X)\cong \mathscr{Rep}_n(\mathsf{QCoh}(X))\cong\Rep_n(A).
\end{align*} 
\end{thm} 
\begin{rmk}
	 By the above theorem, this Quot scheme on smooth and affine $X=\Spec A$ is equivalent to the corresponding open subfunctor of the derived enhancement of this Quot scheme defined in \cite{quot} (see Example 4.1, Proposition 4.2, and Definition 3.3 of \textit{op. cit.}). In fact, these also match when $X$ is smooth quasi-projective by Proposition \ref{prop:matchup-with-adhikari}.
\end{rmk} 
Now, let $X$ be a (not necessarily smooth) quasi-projective scheme. We show that the scheme $\Quot^{n,\text{fr}}_{\mathscr{O}^n_X}(X) \times X$ admits an almost perfect complex $\mathscr{R}$ called the ``universal representation.'' This complex is a generalization of the domain wall which defines the map of Theorem \ref{thm:intro-trace-map}.
\begin{thm}[Proposition \ref{prop:R-is-a-domain-wall} in the text]
	The sheaf $\mathscr{R}$ defines a \textbf{domain wall} between the partially extended field theories defined by
\begin{align*}
	\mathsf{QCoh}(X)\quad\text{and}\quad \mathsf{QCoh}(\Quot^{n,\text{fr}}_{\mathscr{O}_X^n}(X)),
\end{align*}
meaning that the Fourier-Mukai integral transform it defines,
\begin{align*}
	\Phi_{\mathscr{R}}:\mathsf{QCoh}(X)\to &\mathsf{QCoh}(\Quot^{n,\text{fr}}_{\mathscr{O}_X^n}(X))\\
	\mathscr{F}\mapsto &p_{\Quot^{n,\text{fr}}_{\mathscr{O}_X^n}(X)*}(p_X^*\mathscr{F}\otimes \mathscr{R})
\end{align*} 
has a continuous right adjoint. Here, $p_X,p_{\Quot^{n,\text{fr}}_{\mathscr{O}_X^n}(X)}$ are the respective projections from $X\times \Quot^{n,\text{fr}}_{\mathscr{O}_X^n}(X)$.
\end{thm} 
In particular, this theorem implies that in addition to functors between the respective categories, $\mathscr{R}$ induces a map on Hochschild homology by a theorem of \cite{nonlinear-trace}, which is moreover $S^1$-equivariant.
 
\begin{thm}[Theorem \ref{thm:stack-domain-wall} in the text]
	If $X=\Spec A$ is a finitely presented affine scheme, then the domain wall defined by $\mathscr{R}$ precisely recovers the domain wall which gives rise to the trace map of \cite{BKR}.	
\end{thm} 
Thus, this generalizes the story of \cite{BKR} by initiating the study of analogous trace maps for arbitrary classical quasi-projective schemes.

Let $\Coh^n_0(X)$ denote the derived stack of length $n$ torsion sheaves on a classical quasi-projective scheme $X$ (as defined in \cite{PS} and reviewed in \S\ref{sec:trace-maps-II}). There is a canonical projection map
\begin{align*}
	\Quot^{n,\text{fr}}_{\mathscr{O}^n_X}(X)\to \Coh^n_0(X),
\end{align*} 
which informally can be described as sending a quotient $\mathscr{O}^n_X\twoheadrightarrow \mathscr{E}$, where $\mathscr{E}$ is length $n$ torsion, to $\mathscr{E}$ itself. We show (Proposition \ref{prop:quot-scheme-is-torsor}) that this map is a $\GL_n$-torsor; thus, $\Coh^n_0(X)$ is the analog of the derived character stack $\Rep_n(A) /\GL_n$ for non-affine schemes. When $X$ is furthermore smooth, this stack admits a formally \'etale map
\begin{align*}
	f:\Coh^n_0(X)\to \Perf_{\text{prop}}(X)
\end{align*} 
to the stack of perfect complexes with proper support. In the case that $X$ is Calabi-Yau of dimension $d$, this latter stack admits a $(2-d)$-shifted symplectic structure by \cite{BD}.
\begin{thm}[Theorem \ref{thm:symplectic-structure} in the text]
	Let $X$ be a smooth classical quasi-projective scheme equipped with a trivialization $\mathscr{O}_X\cong K_X$ of its canonical bundle. Then $\Coh^n_0(X)$ admits a $(2-d)$-shifted symplectic structure induced from the one on $\Perf_{\text{prop}}(X)$ by pullback along $f$.
\end{thm} 
This theorem has been shown for projective $X$ in \cite{BKSv2}; as far as we know, the extension to quasi-projective $X$ in the context of framed representation moduli is new. Moreover, the above extension to quasi-projective schemes thus defines a shifted symplectic structure on the derived character stack
\begin{align*}
	\Rep_n(A) /\GL_n
\end{align*} 
when $A$ is a smooth commutative algebra of finite presentation and the stack is equipped with a Calabi-Yau structure. We expect that this symplectic structure agrees with the structure obtained in \cite{symplectic-structure} on derived character varieties.

\subsection{Notation and Conventions}\label{sec:notation-and-conventions}
All categories considered are without further comment $\infty$-categories and all functors derived functors/functors between $\infty$-categories \cite{HTT}. For example, $\Hom$ denotes the derived functor $\R \Hom$, $\otimes $ denotes the derived functor $\otimes ^{\mathbb{L}}$, etc. Schemes are likewise assumed derived (see \cite{GR} for an introduction). Any object, map, or functor will be explicitly called ``classical'' if it is supposed to be the underived version. We always work over a fixed field $k$ of characteristic 0, and our gradings will always be cohomological. All algebras are assumed to be unital.

Our running assumptions on stacks are that they are locally almost of finite type (laft) with affine diagonal. The exception is in Section \ref{sec:ext-to-stacks}, where we use the big stacks of quasi-coherent complexes (which are not laft). If $X$ is a scheme or stack, we denote by $X^\text{cl}$ the classical truncation \cite[Chapter 2]{GR}. We use sans serif font for names of categories; e.g., $\mathsf{Perf}(X)$ denotes the $(\infty,1)$-category of perfect complexes on a scheme or stack $X$, and we use ordinary font for names of stacks; e.g., $\Perf(X)$ is the stack of perfect complexes on $X$.

\subsection{Glossary of Categories}
Since $k$ is assumed to be a characteristic $0$ field throughout, the category $\mathsf{CAlg}_k$ (often denoted simply $\mathsf{CAlg}$) is the category of commutative differential graded $k$-algebras. $\mathsf{CAlg}^{\le 0}\subset \mathsf{CAlg}$ is the full subcategory of connective commutative algebras. It is our model for (the opposite to) the category of derived affine schemes, and we write $\mathsf{Aff}$ for its opposite category.

We denote by $\mathsf{Alg}_k$ the category of associative differential graded $k$-algebras.

For $A \in \mathsf{Alg}_k$, we let $\mathsf{LMod}_A\eqcolon A\Mod$ denote the stable category of left $A$-modules. Similarly, $\mathsf{RMod}_A$ denotes the stable category of right $A$-modules. If $A\in \mathsf{CAlg_k}\subset \mathsf{Alg}_k$, then we denote $\mathsf{QCoh}(\Spec A)\coloneqq A\Mod$. The subcategory of compact objects is denoted by $\mathsf{LPerf}_A$; in the case that $A$ is commutative, we write $\mathsf{Perf}(\Spec A)\coloneqq A\text{-}\mathsf{perf}$. The category $A\Mod$ comes equipped with a canonical $t$-structure defined in \cite[Chapter 2 Section 1.2]{GR}.

For a prestack $X$, we let $\mathsf{QCoh}(X)\coloneqq \lim_{\Spec A \to X} A\Mod$, where the limit is taken over all derived affine schemes mapping to $X$. We also have the category $\mathsf{Perf}(X)\coloneqq \lim_{\Spec A \to X}A\text{-}\mathsf{perf}$.

We denote by $\mathsf{Mor}_k$ the Morita $(\infty,2)$-category of $k$-algebras, bimodules, and intertwiners of \cite{haugseng-definition-of-morita-bicat}. Informally, objects of this category are algebras, 1-morphisms $A\to B$ are $(A,B)$-bimodules, and 2-morphisms are intertwining bimodule-homomorphisms.

We denote by $\mathsf{Pr}^\text{L}$ the $(\infty,2)$-category of presentable $(\infty,1)$-categories, colimit preserving functors, and natural transformations as in \cite{HA}.

We let $\mathsf{Cat}$ denote the $(\infty,2)$-category of all $(\infty,1)$-categories, all functors, and natural transformations.

The category of spaces, equivalently, of  $\infty$-groupoids, is denoted $\mathsf{Spc}$. This category admits a natural inclusion $\mathsf{Spc}\hookrightarrow \mathsf{Cat}$.
\subsection{Acknowledgments}
I would like to thank my advisor, David Nadler, for helping me with the TQFT portion of this project and for his continued support, encouragement, and advice during this project. This paper would not have been completed without his support. I would also like to thank Yuri Berest for suggesting the portion on extending representation varieties to stacks, as it led to the most interesting (in my opinion) results. Ansuman Bardalai taught me much of the basic (and not so basic) algebraic geometry that I used during the completion of this project and resolved many of my algebro-geometric confusions related to this project. I also benefited from conversations with Daigo Ito, Kabir Kapoor, and Theo Johnson-Freyd related to this project.

\subsubsection{AI Statement}
ChatGPT 5.5 was used as this project neared completion to locate mathematical and typographic errors. It was also used throughout the project for literature search.	

\section{Domain Walls}\label{sec:domain-walls}
\subsection{Dualizability}
Consider the Morita $(\infty,2)$-category of $k$-algebras, bimodules, and intertwiners, $\mathsf{Mor}_k$, which is defined in e.g. \cite{haugseng-definition-of-morita-bicat}. We will say simply ``2-category'' in the sequel, but we always mean ``$(\infty,2)$-category.'' We now record some basic results implied by \cite[Theorem 5.1]{GS}.
\begin{prop}[Gwilliam-Scheimbauer]
	Every object $A\in\mathsf{Mor}_k$ is dualizable with dual $A^{\text{op}}$, coevaluation ${}_kA_{A\otimes A^{\text{op}}}$, and evaluation ${}_{A^{\text{op}}\otimes A}A_k$.
\end{prop}

Since we're in a 2-category, it is natural to ask for dualizable 1-morphisms in addition to objects. For this, we require unit and counit 2-morphisms which exhibit a given 1-morphism as dualizable \cite{TFT}. For example, a 1-morphism $M:A\to B$ is right dualizable with right dual $M^R$ if there exist unit and counit maps, $\eta:A\to M\otimes _BM^R$ and $\varepsilon :M^R\otimes _AM\to B$, such that the following composites are $\id_M$ and $\id_{M^R}$, respectively \cite[Definition 4.6.1.1]{HA}.
\begin{align*}
	M=A\otimes _AM\xrightarrow{\eta\otimes \id_M}&M\otimes _BM^R\otimes _AM\xrightarrow{\id_M\otimes \varepsilon }M\otimes _BB=M\\
	M^R=M^R\otimes _AA\xrightarrow{\id_{M^R}\otimes \eta}&M^R\otimes _AM\otimes _BM^R\xrightarrow{\varepsilon \otimes \id_{M^R}}B\otimes _BM^R=M^R.
\end{align*}
\begin{lem}\label{lem:eval-exists}
	Let $M$ be an $(A,B)$-bimodule. Then there are canonical evaluation maps
	\begin{align*}
		\Hom_B(M,B)\otimes _A M &\to B\\
		M\otimes _B\Hom_A(M,A)&\to A.
	\end{align*} 
\end{lem} 
\begin{proof}
	The functor
	\begin{align*}
		-\otimes _A M: \mathsf{RMod}_A\to \mathsf{RMod}_B
	\end{align*} 
	has a right adjoint $\Hom_B(M,-)$, and the counit of the adjunction evaluated at $B$ is precisely
	\begin{align*}
		\Hom_B(M,B)\otimes _AM\to B.
	\end{align*} 
	The proof of the existence of the evaluation
	\begin{align*}
		M\otimes _B\Hom_A(M,A)\to A
	\end{align*} 
	is similar.
\end{proof}

\begin{prop}[Lurie]\label{prop:dual}
	Let $M$ be an $(A,B)$-bimodule viewed as a 1-morphism in $\mathsf{Mor}_k$. Then $M$ has a left (resp. right) dual if and only if it is a perfect $A$-module (resp. $B$-module) with left (resp. right) dual $M^L=\Hom_A(M,A)$ (resp. $M^R=\Hom_B(M,B)$).
\end{prop} 
\begin{proof}
	Propositions 4.6.2.1 and 4.6.2.13 of \cite{HA} state that $M$ is left dualizable as an $(A,B)$-bimodule if and only if the restriction of $M$ to an $A$-module is left dualizable. We then observe that dualizable $A$-modules are precisely the perfect $A$-modules by \cite[Proposition 7.2.4.2]{HA}. The existence of the left dual implies that the functor
	\begin{align*}
		-\otimes _A M: \mathsf{RMod}_A\to \mathsf{RMod}_B
	\end{align*} 
	has a left adjoint given by $-\otimes _B M^L$ by Proposition 4.6.2.1 of \textit{op. cit.} Since $M$ is perfect as a left $A$-module, (the dual to) Proposition 7.2.4.4 of \textit{op. cit.} shows that 
	\begin{align*}
		M^L\otimes_A -:\mathsf{LMod}_A\to k\Mod
	\end{align*} 
	is given by the functor corepresented by $M$, i.e. by $\Hom_A(M,-)$. This identifies 
	\begin{align*}
		M^L\otimes _AA\cong M^L=\Hom_A(M,A).
	\end{align*} 
	The proof for $M^R $ is similar.
\end{proof}

	\subsection{Trace Maps I}\label{sec:trace-maps-I}
	By \cite[Section 3.4]{nonlinear-trace}, any continuous (that is, colimit-preserving) morphism $M:A\to B$ of dualizable objects in $\mathsf{Mor}_k$ which is furthermore equipped with a continuous right adjoint defines a map
\begin{align*}
	\mathrm H\mathrm H_*(A)\to \mathrm H\mathrm H_*(B)
\end{align*} 
between the Hochschild homologies of $A,B$. We call such a morphism $M$, which is a dualizable $(A,B)$-bimodule, a \textbf{domain wall} between the categories $\mathsf{RMod}_A$ and $\mathsf{RMod}_B$; i.e., $M$ defines a colimit preserving functor between the categories $\mathsf{RMod}_A$ and $\mathsf{RMod}_B$ which moreover has a colimit preserving right adjoint. For example, we have a functor
\begin{align*}
	M_{\pt}: \mathsf{RMod}_A&\to \mathsf{RMod}_B\\
	K&\mapsto K\otimes _A M.
\end{align*} 
The map
\begin{align*}
	M_{S^1}:\mathrm H\mathrm H_*(A)\to \mathrm H\mathrm H_*(B)
\end{align*} 
is defined as follows. (Right) dualizability ensures we have unit and counit maps
\begin{align*}
	\eta:&A\to M\otimes _BM^R\\
	\varepsilon:&M^R\otimes _AM\to B,
\end{align*} 
where $M^R=\Hom_B(M,B)$ is the right dual of $M$. We thus obtain a map of $k$-vector spaces
\begin{align}\label{eq:HH-map}
	\mathrm H\mathrm H_*(A)\to A\underset{A\otimes A^{\text{op}}}{\otimes }M\otimes _BM^R=B\underset{B\otimes B^{\text{op}}}{\otimes }M^R\otimes _AM\to \mathrm H\mathrm H_*(B).
\end{align} 
Moreover, this map admits a canonical $S^1$-equivariant refinement by \cite[Theorem 2.14]{HSS}.

We claim that the trace map studied in \cite{BKR} exactly arises from such a domain wall construction.

\begin{lem}\label{lem:A-module-structure}
	Let $A\in \mathsf{Alg}$ be an associative algebra over $k$, and let $QA\to A$ be a semi-free dg, i.e. cofibrant, resolution of $A$. Denote by $B=\mathbb{L}(A)_n$ its derived $n$-dimensional representation homology, i.e. the left derived functor of $(-)_n$ applied to $A$, and let $M=B^n$. Then $M$ is a $(QA,B )$-bimodule.
\end{lem} 
\begin{proof}
	The unit of the adjunction $((-)_n,M_n)$ is a map of algebras $QA\to M_n(B) = \End_B(B^n)$, which is a $(QA,B)$-bimodule structure on $B^n$.
\end{proof} 
\begin{rmk}
	 Since the quasi-isomorphism $QA\to A$ induces a Morita equivalence, we see that $M$ has an $(A,B)$-bimodule structure.
\end{rmk}

Since $M$ is a perfect $B$-module, it is right dualizable and hence defines a domain wall between the B-models associated to $QA$ and $B$. By Proposition \ref{prop:dual} it has a right dual given by $M^R=\Hom_B(M,B)$, and
\begin{align*}
	M\otimes _BM^R\cong\End_B(M)\cong M_n(B).
\end{align*}
\begin{prop}\label{prop:domain-wall-map-is-counit}
	In the setting of Lemma \ref{lem:A-module-structure}, the map
\begin{align*}
	QA \otimes _{QA\otimes QA^{\text{op}}}QA\to QA\otimes _{QA\otimes QA^{\text{op}}}M\otimes _BM^R
\end{align*}
induced by the domain wall given by $M$ is exactly $\id_{QA}\otimes \eta$, where $\eta$ is the unit of the adjunction $((-)_n,M_n)$. 
\end{prop}
\begin{proof}
	$M$ is right dualizable with right dual $M^R$. Thus, $M^R$, viewed as a 1-morphism in the Morita 2-category, has unit given by
	\begin{align*}
		QA\to M\otimes _BM^R,
	\end{align*} 
	and this is the unit of the adjunction
	\begin{align*}
		-\otimes _{QA}M\dashv-\otimes _BM^R.
	\end{align*} 
	Under the isomorphism
	\begin{align*}
		M\otimes _BM^R\cong \End_B(M),
	\end{align*} 
	this unit is the map
	\begin{align*}
		QA\to \End_B(M)
	\end{align*} 
	which defines the left $QA$-module structure on $\End_B(M)$. On the other hand, the left $QA$-action on
	\begin{align*}
		\End_B(M)\cong M_n((QA)_n)
	\end{align*} 
	is induced by the unit of the adjunction $((-)_n,M_n)$,
	\begin{align*}
		\eta_{QA}:QA\to M_n((QA)_n)\cong\End_{B}(M).
	\end{align*} 
	Thus, we see that the units of the two adjunctions in question are the same, so the map in the statement of the Proposition is $\id_{QA}\otimes \eta_{QA}$.
\end{proof}

Let $B\in\mathsf{CAlg}$ be a commutative algebra, and let $\Map(S^1,\Spec B)\eqcolon \mathscr{L}\Spec B$ be the loop space of $X=\Spec B$. There is a canonical inclusion map $X\to \mathscr{L}X$ given by constant loops, which corresponds to a map of algebras
\begin{align}\label{eq:const-loops}
	\mathrm H\mathrm H_*(B)=\mathscr{O}(\mathscr{L}\Spec (B))\to B.
\end{align} 
Further, this map is equivariant with respect to the $S^1$-actions on $\mathscr{L}X$ and $X$, where $S^1$ acts on $X$ trivially.
\begin{prop}\label{prop:natural-transf-agree-on-truncations}
	Let $\mathsf{Alg}^{\le 0}\subset \mathsf{Alg}$ denote the full subcategory of connective associative algebras. Consider the functors
	\begin{align*}
		\mathrm {HC}_*,\mathbb{L}(-)_n: \mathsf{Alg}^{\le 0}\to \mathsf{Vect},
	\end{align*} 
	where in this case $\mathbb{L}(A)_n$ denotes the underlying vector space of the connective commutative algebra. Suppose given natural transformations
	\begin{align*}
		\alpha,\beta:\mathrm{HC}_*\Rightarrow  \mathbb{L}(-)_n
	\end{align*} 
	such that for every finite-dimensional vector space $V$ with tensor algebra $\mathscr{T}(V)\coloneqq \bigoplus_n V^{\otimes n}\in \mathsf{Alg}^{\le 0}$, the maps
	\begin{align*}
		\mathrm{H}^0(\mathrm{HC}_*(\mathscr{T}(V)))\to (\mathscr{T}(V))_n
	\end{align*} 
	induced by $\alpha_{\mathscr{T}(V)}$ and $\beta_{\mathscr{T}(V)}$ agree. Then 
	\begin{align*}
		\alpha\cong\beta.
	\end{align*} 
\end{prop} 
\begin{proof}
	Both $\mathrm{HC}_*$ and $\mathbb{L}(-)_n$ preserve sifted colimits. Indeed, Hochschild homology is computed via the geometric realization of the cyclic bar construction, whose composition preserves sifted colimits. Similarly, $S^1$-coinvariants is a colimit; hence, $\mathrm {HC}_*$ preserves sifted colimits. On the other hand, the functor $\mathbb{L}(-)_n$ is a left adjoint, and the forgetful functor $\mathsf{CAlg}^{\le 0}\to \mathsf{Vect}$ also preserves sifted colimits.

	Let $\mathsf{Alg}^{0}\subset \mathsf{Alg}^{\le 0}$ denote the subcategory spanned by algebras of the form $\mathscr{T}(V)$ with $V$ a finite-dimensional vector space concentrated in degree 0. By the remarks after Corollary 7.1.4.17 of \cite{HA}, the category $\mathsf{Alg}^{\le 0}\cong \mathscr{P}_\Sigma(\mathsf{Alg}^0)$, which means that every connective associative algebra can be expressed as a sifted colimit of algebras of the form $\mathscr{T}(V)$. Thus, it is sufficient to show that $\alpha\cong\beta$ on this full subcategory. Since $\mathscr{T}(V)$ is cofibrant for any $V$, we see that
	\begin{align*}
		\mathbb{L}(\mathscr{T}(V))_n=(\mathscr{T}(V))_n.
	\end{align*} 
	Moreover, the underlying vector space of this algebra is in the heart of $\mathsf{Vect}$. On the other hand, $\mathrm{HC}_*(\mathscr{T}(V))$ is connective, so that we can identify
	\begin{align*}
		\Hom(\mathrm{HC}_*(\mathscr{T}(V)),(\mathscr{T}(V))_n)\cong \Hom^0(\mathrm H^0(\mathrm{HC}_*(\mathscr{T}(V))),(\mathscr{T}(V))_n),
	\end{align*} 
	where the RHS is the inner Hom in $\mathsf{Vect}^{\heartsuit}$; in particular, it is a discrete space. Since we also have
	\begin{align*}
		\mathrm H^0(\mathrm{HC}_*(\mathscr{T}(V)))=\mathscr{T}(V) /[\mathscr{T}(V),\mathscr{T}(V)],
	\end{align*} 
	we see that a point of $\Hom(\mathrm{HC}_*(\mathscr{T}(V)), (\mathscr{T}(V))_n)$ is uniquely determined by a map
	\begin{align*}
		\mathscr{T}(V) /[\mathscr{T}(V),\mathscr{T}(V)]\to (\mathscr{T}(V))_n.	
	\end{align*} 
	Now, the space of natural transformations between $\mathrm{HC}_*$ and $\mathbb{L}(-)_n$ is a limit of the mapping spaces $$\Hom(\mathrm{HC}_*(\mathscr{T}(V)),(\mathscr{T}(W))_n),$$ which are all discrete because the source is connective and the target is in the heart. Thus, the limit is 0-truncated, so the above identification of points is sufficient to identify the natural transformations. Since by assumption the maps induced by $\alpha_{\mathscr{T}(V)}$ and $\beta_{\mathscr{T}(V)}$ agree and are of this form, we conclude that
	\begin{align*}
		\alpha\cong\beta.
	\end{align*} 
\end{proof} 
Let $A \in \mathsf{Alg}^{\le 0}$ be a connective associative algebra, and set
\begin{align*}
	B\coloneqq \mathbb{L}(A)_n.
\end{align*} 
Denote the unit of the adjunction $(\mathbb{L}(-)_n,M_n)$ by $\eta:A\to M_n(B)$, and denote $M\coloneqq B^n$ the $(A,B)$-bimodule of Lemma \ref{lem:A-module-structure}. Since $M$ is perfect as a right $B$-module, it is right dualizable and hence induces, by \cite[Theorem 2.14]{HSS}, a canonical $S^1$-equivariant map
\begin{align*}
	\mathrm H\mathrm H_*(A)\to \mathrm H\mathrm H_*(B).
\end{align*} 
Composing this with the canonical map $\mathrm H\mathrm H_*(B)\to B$ given by (\ref{eq:const-loops}), we obtain an $S^1$-equivariant map
\begin{align}\label{eq:trace}
	\mathrm H\mathrm H_*(A)\to \triv(B),
\end{align} 
where $\triv(B)$ denotes the underlying vector space of $B$ equipped with the trivial $S^1$-action.

\begin{thm}\label{thm:trace}
	Under the adjunction 
	\begin{align*}
		(-)_{S^1}: \mathsf{Vect}^{S^1}\to \mathsf{Vect}:\triv,
	\end{align*} 
	the map
	\begin{align*}
		\mathrm{HC}_*(A)\to B
	\end{align*} 
	induced by (\ref{eq:trace}) is the trace map of \cite{BKR}.
\end{thm} 
\begin{proof}
	We first claim that the map 
	\begin{align*}
		t_A:\mathrm{HC}_*(A)\to B
	\end{align*} 
	induces a natural transformation
	\begin{align*}
		t:\mathrm{HC}_*\Rightarrow \mathbb{L}(-)_n.
	\end{align*} 
	Indeed, by applying to $\eta$ the functor $\mathsf{Alg}\to \mathsf{Mor}$ which sends maps to base change bimodules, we see that the assignment $A\leadsto (A\xrightarrow{M}\mathbb{L}(A)_n)$ assembles, by functoriality of extension of scalars, into a coherent family of 1-morphisms $A\to M_n(B)$, because in $\mathsf{Mor}$, we have the Morita equivalence $M_n(B)\cong B$ which is natural in $B$ (since $M_n(B)\cong M_n(k)\otimes B$). Furthermore, by naturality of the  trace functor of \cite[Theorem 2.14]{HSS}, we thus obtain an $S^1$-equivariant natural transformation
	\begin{align*}
		\mathrm H\mathrm H_*\to \mathbb{L}(-)_n,
	\end{align*} 
	which thus induces the natural transformation $t$. The trace map $\theta$ of \cite{BKR} is also a natural transformation by the remarks around Equations 4.3 and 4.4 of \textit{op. cit.}; moreover, these are natural transformations with the same source and target, since by \cite[Proposition 4.24 and Remark 6.12]{HSS}, the $S^1$-structure on the sources are the same, functorially in $A$. Thus, by Proposition \ref{prop:natural-transf-agree-on-truncations}, it suffices to show that the two natural transformations $t$ and $\theta$ agree on zeroth cyclic homology of algebras of the form $A=\mathscr{T}(V)$ with $V$ a finite-dimensional vector space. Letting $B=(\mathscr{T}(V))_n=\mathbb{L}(\mathscr{T}(V))_n$ and $\eta: \mathscr{T}(V)\to M_n(B)$, we have the associated bimodule $M=B^n$. We can now work element by element. Let $e_1,\ldots,e_n$ be the standard basis of $M$ with dual basis $e_1^\vee,\ldots,e_n^\vee$ in $M^R=\Hom_B(M,B)$. The map 
	\begin{align*}
	 \mathscr{T}(V) /[\mathscr{T}(V),\mathscr{T}(V)]\to (\mathscr{T}(V))_n
	\end{align*}
	induced by $M$ is explicitly
	\begin{align*}
		\bar{a}\mapsto \sum e_i^\vee(\eta(a)e_i),
	\end{align*} 
	which is just the ordinary trace
	\begin{align*}
		\sum e_i^\vee(\eta(a)e_i)=\tr(\eta(a)).
	\end{align*} 
	Indeed, we can calculate this at the chain level as follows. We need only identify the map in question on $\mathrm H^0$. There, the first map sends
	\begin{align*}
		A /[A,A]\to &\mathrm H^0(A\otimes _{A\otimes A^{\text{op}}}M_n(B))=M_n(B) /[A,M_n(B)]\\
		\bar{a}\mapsto &\eta(a).
	\end{align*} 
	 This is then sent to
	 \begin{align*}
		 M_n(B) /[A,M_n(B)]\to& \mathrm H^0(A\otimes _{A\otimes A^{\text{op}}}M\otimes _B\Hom_B(M,B))\\
		 \eta(a)\mapsto &\sum\eta(a)e_i\otimes e_i^\vee.
	 \end{align*} 
	 	 Now, under the isomorphism
	 \begin{align*}
	 	A\otimes _{A\otimes A^{\text{op}}}M\otimes _BM^R\cong M^R\otimes _AM\otimes _{B\otimes B}B,
	 \end{align*} 
	 we see that 
	\begin{align*}
		\sum\eta(a)e_i\otimes e_i^\vee\mapsto \sum e_i^\vee\otimes \eta(a)e_i.
			\end{align*} 
	Finally, this is sent via $\varepsilon\otimes \id_B$, where $\varepsilon$ is the counit duality datum for $M$, to the element
	\begin{align*}
		\sum e_i^\vee(\eta(a)e_i),
	\end{align*} 
	as desired.

	Moreover, since $B$ is commutative, $B /[B,B]=B$, so the constant loop map on zeroth Hochschild homology is the identity. Thus, $t_{\mathscr{T}(V)}$ induces the map
	\begin{align*}
		\bar{a}\mapsto \tr(\eta(a)),
	\end{align*} 
	which is precisely the definition of the Berest-Khachatryan-Ramadoss trace, as in e.g. \cite[Section 3]{BR}.
\end{proof}

\section{Extension to Stacks}\label{sec:ext-to-stacks}
	
\subsection{Definition}
Let $ \mathcal{C}$ be a category. Note that $\mathcal{C}$ defines a moduli functor $\mathcal{X}_{\mathcal{C}}$, which is the composition of the functor
\begin{align*}
	\mathsf{Aff}^{\text{op}}\to& \mathsf{Cat}\\
	S\mapsto &\mathcal{C}\otimes \mathsf{QCoh}(S)
\end{align*} 
with the right adjoint to the inclusion $\mathsf{Spc}\hookrightarrow \mathsf{Cat}$. Informally, this functor can be described as sending a map of affines $f:S'\to S$ to 
\begin{align*}
\ywidetilde{\mathcal{C}\otimes \mathsf{QCoh}(S)}\xrightarrow{\id_{\mathcal{C}}\otimes f^*}\ywidetilde{\mathcal{C}\otimes \mathsf{QCoh}(S')},
\end{align*}
where for a category $\mathscr{A}$, $\ywidetilde{\mathscr{A}}$ denotes the maximal subgroupoid of $\mathscr{A}$.
\begin{rmk}
This functor has as a subfunctor the perhaps more familiar moduli functor of pseudo-perfect complexes $\mathscr{M}_{\mathcal{C}}$ of \cite{TV}.
\end{rmk}
\begin{prop}
	If $\mathcal{C}=\mathsf{QCoh}(X)$, the category of quasi-coherent sheaves on a prestack $X$, or $\mathcal{C}=A\Mod$ is modules over an associative algebra $A$, then $\mathcal{X}_{\mathcal{C}}$ is a stack.
\end{prop} 
\begin{proof}
Let $A$ be an associative algebra and $\mathcal{C}=A\Mod$. Given a faithfully flat map of commutative algebras $B\to C$, we can form its Amitsur complex $(C /B)_\bullet$, which is by definition the underlying cosimplicial algebra of the \v{C}ech nerve of $\Spec C \to \Spec B$. To prove descent for $\mathcal{X}_{A\Mod}$, it thus suffices to provide an equivalence
\begin{align*}
	\lim(A\Mod\otimes (C/B)_\bullet\Mod)\cong A\Mod \otimes (\lim(C /B)_\bullet)\Mod.
\end{align*} 
By faithfully flat descent, we have
\begin{align*}
	(\lim(C /B)_\bullet)\Mod\cong \lim((C /B)_\bullet\Mod).
\end{align*} 
Hence, it suffices to show that the Lurie tensor product with $A\Mod$ commutes with this limit. By construction, this tensor can be identified with colimit preserving functors between stable categories, since $A\Mod$ is dualizable in $\mathsf{Pr}^\text{L}$ (with dual $A\Mod^\vee=A^{\text{op}}\Mod$):
\begin{align*}
	A\Mod\otimes (\lim(C /B)_\bullet)\Mod\cong \mathsf{Map}_{\mathsf{Pr}^\text{L}}((A\Mod)^\vee,(\lim(C /B)_\bullet)\Mod).
\end{align*} 
Since the $\mathsf{Map}$ functor commutes with limits in the second variable, we have the result.

If $X$ is a prestack, then for $\mathcal{C}=\mathsf{QCoh}(X)$, $\mathcal{X}_{\mathcal{C}}$ is given by
\begin{align*}
	S\mapsto \ywidetilde{\mathsf{QCoh}(X\times S)},
\end{align*} 
which is the composition of the functor $\ywidetilde{\mathsf{QCoh}(-)}$ with $X\times -$ by \cite[Chapter 3 Proposition 3.5.3]{GR}. Now, given a cover $S'\to S$, $X\times -$ applied to the \v{C}ech nerve of this map is also a \v{C}ech nerve (of the map $X\times S'\to X\times S$). Since the functor $\mathsf{QCoh}(-)$ satisfies descent and passing to the maximal subgroupoid is a right adjoint, the composite $\mathcal{X}_{\mathcal{C}}$ satisfies descent, too.
\end{proof} 
\begin{lem}\label{lem:perf-pushforward}
	Let $\Vect\coloneqq \mathcal{X}_{k\Mod}$ denote the stack of quasi-coherent complexes, so $\Vect(S)=\mathsf{QCoh}(S)^{\cong}$. If $\mathcal{C}=\mathsf{QCoh}(X)$ is the category of quasi-coherent sheaves on a perfect stack or if $\mathcal{C}=A\Mod$ is modules over an associative algebra $A\in \mathsf{Alg}_k$, we have a natural transformation
\begin{align*}
	\mathcal{X}_{\mathcal{C}}\to \Vect.
\end{align*} 
\end{lem} 
\begin{rmk}
	In the case of $\mathcal{C}=A\Mod$, this natural transformation \textit{depends on the presentation} of $\mathcal{C}$ as a module category. When we write $\mathcal{C}=A\Mod$, we always implicitly choose such a presentation.
\end{rmk} 
\begin{proof}
Let $\mathcal{C}=A\Mod$ be modules over an associative algebra. We claim restriction of modules defines a natural transformation
\begin{align*}
	\mathcal{X}_{\mathcal{C}}\to \Vect.
\end{align*} 
Given a map $B\to C$ of algebras, we need to show that the following diagram commutes, where the horizontal arrows are restriction and the vertical arrows base change.
\[
\begin{tikzcd}
	(A\otimes B)\Mod\arrow[r]\arrow[d]&B\Mod\arrow[d]\\
	(A\otimes C)\Mod\arrow[r]&C\Mod
\end{tikzcd}.
\] 
Indeed, this follows from the isomorphism
\begin{align*}
	(A\otimes C)\otimes _{A\otimes B}M\cong C\otimes _BM.
\end{align*} 

Now, let $\mathcal{C}=\mathsf{QCoh}(X)$ for $X$ a perfect stack. Consider the assignment
\begin{align*}
	\mathcal{X}_{\mathcal{C}}(S)=\mathsf{QCoh}(X\times S)^{\cong}\to& \Vect(S)=\mathsf{QCoh}(S)^{\cong}\\
	\mathscr{F}\mapsto &p_{*}\mathscr{F},
\end{align*} 
where $p:X\times S\to S$ is the projection. We claim that the assignment $\mathscr{F}\mapsto p_*\mathscr{F}$ is natural in $S$. Explicitly, given $f:S\to S'$ a map of affines, we need to check that the following square commutes
\[
\begin{tikzcd}
	\mathcal{X}_{\mathcal{C}}(S')\arrow[r]\arrow[d]&\mathcal{X}_{\mathcal{C}}(S)\arrow[d]\\
	\mathsf{QCoh}(S')^{\cong}\arrow[r]&\mathsf{QCoh}(S)^{\cong}
\end{tikzcd}.
\] 
This is equivalent to checking base change for the square
\[
\begin{tikzcd}
	X\times S\arrow[r,"\id_X\times f"]\arrow[d,"p"]&X\times S'\arrow[d,"p'"]\\
	S\arrow[r,"f"]&S'
\end{tikzcd}.
\] 
Since $S'$ and $X\times S'$ are perfect, so is the map $p'$. Hence, base change holds by \cite[Proposition 3.10]{BFN}. 
\end{proof}

\begin{rmk}\label{rmk:map-on-perfect-stacks}
	If $X$ is a qcqs scheme, then,  as in \cite[\S8.3]{T06}, we can present
	\begin{align*}
		\mathcal{X}_{\mathsf{QCoh}(X)}(S) \cong\ywidetilde{\mathsf{QCoh}(X\times S)}\cong\Map(\mathsf{QCoh}(X),\mathsf{QCoh}(S)),
	\end{align*} 
	and a sheaf $\mathscr{E}\in \mathsf{QCoh}(X\times S)$ corresponds to the map
	\begin{align*}
		\varphi_\mathscr{E}:\mathscr{F}\mapsto p_{S*}(p_X^*\mathscr{F}\otimes \mathscr{E}).
	\end{align*} 
	In particular, by applying $\varphi_\mathscr{E}$ to the structure sheaf $\mathscr{O}_X$, we recover the map defined in the lemma above. In fact, by \cite[Theorem 4.14]{BFN}, this is the case for any perfect stack $X$.
\end{rmk}

We henceforth only consider categories $\mathcal{C}$ of the type considered in Lemma \ref{lem:perf-pushforward}. For a category $\mathcal{C}$ of the form $\mathsf{QCoh}(X)$ with $X$ a perfect stack or $A\Mod$ with $A$ an associative algebra, we define its \textbf{${n}$-dimensional representation moduli} as
\begin{align*}
	\mathscr{Rep}_n(\mathcal{C})\coloneqq \mathcal{X}_{\mathcal{C}}\times _{\Vect}\{k^n\} ,
\end{align*} 
where the map $\mathcal{X}_{\mathcal{C}}\to \Vect$ is the map of Lemma \ref{lem:perf-pushforward} and $\pt\to \Vect$ classifies $k^n$.
\begin{rmk}
	This definition depends not only on the abstract category $\mathcal{C}$, but also on the map to $\Vect$ defined either by Lemma \ref{lem:perf-pushforward} or by the presentation $\mathcal{C}=A\Mod$ (with its natural restriction functor to $\Vect$). In particular, this definition is \textit{not} Morita invariant.
\end{rmk}

\begin{thm}\label{thm:rep-on-affines}
	If $\mathcal{C}=A\Mod$ with $A$ an associative algebra, then $\mathscr{Rep}_n(\mathcal{C})\cong\Rep_n(A)$ is the derived representation variety of $A$.
\end{thm} 
\begin{proof}
	Let $QA\xrightarrow{\sim }A$ be a cofibrant replacement of $A$ in associative algebras. Since $A\Mod\cong QA\Mod$, we also have
	\begin{align*}
		\ywidetilde{(A\otimes B)\Mod}\cong\ywidetilde{(QA\otimes B)\Mod}.	
	\end{align*} 
	Moreover, this equivalence respects pushforward of modules \cite[Theorem 3.3.1]{Hin97}; thus, we are free to assume $A=QA$ is cofibrant.

	In \cite[\S4.7.1]{HA}, Lurie shows that given a category $\mathcal{C}$ and an object $M\in \mathcal{C}$, there is an equivalence
	\begin{align*}
		\mathsf{Map}_{\mathsf{Alg}(\mathcal{C})}(A,\End(M))\cong \mathsf{LMod}_A(\mathcal{C})\times _{\mathcal{C}} \{M\} 
	\end{align*} 
	for any $A\in \mathsf{Alg}(\mathcal{C})$, where $\End(M)$ is the ``universal endomorphism algebra of $M$.'' By construction, this is natural with respect to colimit preserving functors $F: \mathcal{C}\to \mathcal{D}$. Now, consider the categories 
	\begin{align*}
		\mathcal{C}=&{ B\Mod}\\
		\mathcal{D}=&{C\Mod}	,
	\end{align*} 
	for $B,C$ $k$-algebras,	the object $M=B^n$, the algebra $A\otimes B\in \mathsf{Alg}(\mathcal{C})=\mathsf{Alg}_{B /}$, and the functor $F=f^*: \mathcal{C}\to \mathcal{D}$ for a map of algebras $f:B\to C$. By construction, the universal endomorphism algebra is just usual endomorphisms in this case. Thus, Lurie's equivalence becomes an  identification
	\begin{align*}
		\Hom_{\mathsf{Alg}}(A,M_n(B))\cong \mathsf{Map}_{\mathsf{Alg}(\mathcal{C})}(A\otimes B,\End(B^n))\cong \mathcal{X}_{A\Mod}(\Spec B)\times _{\Vect(\Spec B)}\{B^n\} ,
	\end{align*} 
	where the LHS of his equivalence is identified with $\Rep_n(A)(\Spec B)$ by the tensor-hom adjunction. It remains to check that under tensor-hom, the functor $f^*$ on mapping spaces identifies with the map 
	\begin{align*}
		M_n(f):\Hom_{\mathsf{Alg}}(A,M_n(B))\to \Hom_{\mathsf{Alg}}(A,M_n(C)).
	\end{align*} 
	In other words, we just need to check that tensor-hom is natural with respect to maps $f:B\to C$. This would provide an identification of the $S$-points of both $\Rep_n(A)$ and $\mathscr{Rep}_n(A\Mod)$, together with naturality in the test affine $S$. Indeed, given
	\begin{align*}
		\varphi:A\otimes B\to M_n(B),
	\end{align*} 
	$f^*\varphi$ is sent, under tensor-hom, to $f^*\varphi(-\otimes 1_C)$. On the other hand, the canonical isomorphism $M_n(B)\otimes _BC=M_n(C)$ is obtained by sending $M\otimes c\mapsto (f(M_{ij})c)$, where $M \in M_n(B)$ is a matrix and $M_{ij}\in B$ are the entries. Thus, for any $a\in A$, $f^*\varphi(a\otimes 1_C)=M_n(f)(\varphi(a\otimes 1_B))$, as required.
\end{proof}

\subsection{Examples}

\begin{ex}\label{ex:P^1-example}
	Let $\mathcal{C}=\mathsf{QCoh}(\mathbb{P} ^1)$. Then points of $\mathscr{Rep}_n(\mathcal{C})(k)$ are pairs of sheaves $\mathscr{F}\in \mathsf{QCoh}(\mathbb{P} ^1)$ and isomorphisms
	\begin{align*}
		\varphi:\Gamma(\mathbb{P} ^1,\mathscr{F})\xrightarrow{\cong}k^n.
	\end{align*} 
	If $\mathscr{F}$ is coherent, then we can write 
	\begin{align*}
		\mathscr{F}=\bigoplus_i \mathscr{O}(d_i)\oplus \bigoplus _{j}\mathscr{O}(e_j)[1]\oplus T \oplus \bigoplus _{k} \mathscr{O}(-1)[m_k],
	\end{align*} 
	where $d_i\ge 0$, $e_j\le -2$, $T$ is a torsion sheaf of length $\ell$, $\sum(d_i+1) + \sum(-e_j-1)+\ell=n$, and the $m_k$ are arbitrary. For families, points of $\mathscr{Rep}_n(\mathcal{C})(S)$ are pairs of a sheaf $\mathscr{F}\in \mathsf{QCoh}(\mathbb{P} ^1_S)$ and an isomorphism
	\begin{align*}
		\varphi:p_{S*}\mathscr{F}\xrightarrow{\cong}\mathscr{O}_S^n,
	\end{align*} 
	where $p_{S}:\mathbb{P} ^1_S\to S$ is the projection. Over geometric points $p$ of $S$ and $\mathscr{F}$ coherent, $\mathscr{F}|_{\mathbb{P} ^1_{S,p}}$ is as before, except the sheaves $\mathscr{O}(d_i),  \mathscr{O}(e_j),  \mathscr{O}(-1)$ are relative to $S$ and $T$ is finite flat of rank $\ell$ over $S$. Note that due to the appearance of arbitrarily many summands of $\mathscr{O}(-1)$ in $\mathscr{F}$, there is no reason to expect $\mathscr{Rep}_n(\mathcal{C})$ to be quasi-compact.
\end{ex} 
\begin{ex}
	In this example, we will assume that $k$ is algebraically closed. Let $\mathcal{C}=\mathsf{QCoh}(BT)$ for $T$ a torus. Then points of $\mathscr{Rep}_n(\mathcal{C})(k)$ are arbitrary $T$-representations $V$ together with isomorphisms $\varphi:V^T\xrightarrow{\cong}k^n$, where $V^T$ denotes $T$-invariants of the $T$-representation $V$. Recall that $\mathsf{QCoh}(BT)$ decomposes as
	\begin{align*}
		\mathsf{QCoh}(BT)=\mathsf{Rep}(T)\cong \bigoplus _{\lambda\in X^*(T)}\mathsf{Vect}_{\lambda},
	\end{align*} 
	where $X^*(T)$ is the character lattice and $\mathsf{Vect}_{\lambda}$ denotes representations $V$ where $T$ acts by the character $\lambda$. In particular, given a representation $V$,
	\begin{align*}
		V=\bigoplus _{\lambda\in X^*(T)}V_\lambda,
	\end{align*} 
	and $V^T=V_0$ is the 0-weight space. Thus, $k$-points are given by $T$-representations $V$ with the $\lambda=0$ weight having multiplicity $n$ together with a point of $\GL_n(k)$ which identifies $V_0\xrightarrow{\cong} k^n$.

	When we pass to $S$-points, we now have families of $T$-representations $\mathscr{F}\in \mathsf{QCoh}(BT\times S)$ with a trivialization $p_{S*}\mathscr{F}\xrightarrow{\cong}\mathscr{O}^n_S$, where $p_S:BT\times S\to S$, i.e. an element of $\GL_n(S)$. As in the case $S=\pt$, we have a decomposition
	\begin{align*}
		\mathsf{QCoh}(BT\times S)\cong\mathsf{QCoh}(BT)\otimes \mathsf{QCoh}(S)\cong \bigoplus _{\lambda\in X^*(T)}\mathsf{QCoh}(S)_\lambda,
	\end{align*} 
	where $\mathsf{QCoh}(S)_\lambda\coloneqq \mathsf{Vect}_\lambda\otimes \mathsf{QCoh}(S)$ denotes sheaves on $S$ whose underlying vector space is a $T$-representation with weight $\lambda$. Thus, we see that an $S$-point is given by a sheaf $\mathscr{F}\in \mathsf{QCoh}(S)$ with a decomposition
	\begin{align*}
		\mathscr{F}\cong \bigoplus _{\lambda\in X^*(T)}\mathscr{F}_\lambda,
	\end{align*} 
	together with an isomorphism $\mathscr{F}_0\xrightarrow{\cong}\mathscr{O}^n_{S}$, which is an element of $\GL_n(S)$. If we restrict to coherent, rather than quasi-coherent, sheaves (cf. Section \ref{sec:coherent-and-quot-reps}), then we see that the classical truncation of the coherent moduli is given by
	\begin{align*}
		\coprod_{m:X^*(T)\setminus \{0\} \to \mathbb{N} }\prod_{\lambda\neq 0 }B\GL_{m_{\lambda}},
	\end{align*} 
	where $m_\lambda=\rk(\mathscr{F}_\lambda)$ and $m$ has finite support.
\end{ex} 
\begin{ex}
As in the previous example, we still assume that $k$ is algebraically closed. Let $\mathcal{C}=\mathsf{QCoh}(BG)$ for $G$ a reductive group. A $k$-point of the moduli is a pair $(V,\varphi)\in \mathscr{Rep}_{n}(\mathsf{QCoh}(BG))(k)$ with $V$ a $G$-representation equipped with an isomorphism $\varphi:V^G\coloneqq \Hom_G(k,V)\xrightarrow{\cong}k^n$. There is a decomposition
\begin{align*}
	\mathsf{Rep}(G)\cong\prod_{\rho}\mathsf{Vect},
\end{align*} 
where the product is over irreducible representations $\rho$ of $G$, because $k$ is algebraically closed of characteristic 0 and $G$ is reductive. Thus, the representation $V$ decomposes as
\begin{align*}
	V\cong k^n\otimes \rho_{\text{triv}}\oplus \bigoplus _{\rho\neq \rho_{\text{triv}}}M_\rho\otimes \rho,
\end{align*} 
where the first summand is the trivial isotype and the other summands are irreducible representations $\rho$ with multiplicities $M_\rho$.

Passing to $S$-points, as in the case of $BT$, we have families of $G$-representations $\mathscr{F}\in \mathsf{QCoh}(BG\times S)\cong \mathsf{QCoh}(BG)\otimes \mathsf{QCoh}(S)$ together with isomorphisms $\varphi:p_{S*}\mathscr{F}\xrightarrow{\cong}\mathscr{O}^n_{S}$. As in the case of $k$-points, we have an isotypic decomposition
\begin{align*}
	\mathscr{F}\cong \bigoplus _\rho \mathscr{M}_\rho\otimes \rho,
\end{align*} 
where the $\mathscr{M}_\rho\in \mathsf{QCoh}(S)$ and the isomorphism $\varphi$ induces $\mathscr{M}_{\rho_{\text{triv}}}\xrightarrow{\cong}\mathscr{O}^n_S$. If we restrict to coherent sheaves---$G$-equivariant vector bundles on $S$, then the classical truncation of the moduli has a presentation as
\begin{align*}
	\coprod_{m}\prod_{\rho\neq\rho_{\text{triv}}}B\GL_{m_\rho},
\end{align*} 
where the coproduct is over all assignments with finite support of multiplicities to irreducible representations $\rho$.
\end{ex}

\section{Trace Maps II}\label{sec:trace-maps-II}
In Theorem \ref{thm:trace}, we have constructed an $S^1$-equivariant trace map
\begin{align*}
	\mathrm H\mathrm H_*(A) \to \mathbb{L}(A)_n,
\end{align*} 
which corresponds by the universal property of $S^1$-coinvariants to a trace map
\begin{align*}
	\mathrm{HC}_*(A)\to \mathbb{L}(A)_n,
\end{align*} 
arising from a domain wall between the field theories defined by $\mathsf{RMod}_A$ and $\mathsf{RMod}_{\mathbb{L}(A)_n}$. We claim that we have such a domain wall for representation moduli defined by the ``universal representation.'' More precisely, given a quasi-projective classical scheme $X$, we will show that there is a Fourier-Mukai integral transform from $\mathsf{QCoh}(X)$ to 
\begin{align*}
	\mathsf{QCoh}(\Quot^{n,\text{fr}}_{\mathscr{O}_X^n}(X)),
\end{align*} 
 where $\Quot^{n,\text{fr}}_{\mathscr{O}_X^n}(X)$ is a derived enhancement of the framed locus of the Quot scheme of points on $X$ relative to $\mathscr{O}_X^n$. This scheme will be constructed as a substack of ``coherent representations:''
 \begin{align*}
 	\mathscr{Rep}_n^{\text{coh}}(X)\coloneqq \Coh_{\text{prop}}(X)\times _{\Vect}\pt \subset \mathscr{Rep}_n(X),
 \end{align*} 
 which is the substack of $\mathscr{Rep}_n(\mathsf{QCoh}(X))$ of coherent sheaves with proper support (as in \cite{PS}). This integral transform will formally induce a map on Hochschild homologies, which, composed with the constant loop map
 \begin{align*}
 \Quot^{n,\text{fr}}_{\mathscr{O}_X^n}(X)\to 	\mathscr{L}\Quot^{n,\text{fr}}_{\mathscr{O}_X^n}(X),
 \end{align*} 
 will give a generalization of the trace map of \cite{BKR} (see Theorem \ref{thm:stack-domain-wall}).
 \subsection{Coherent Representation Moduli and Derived Quot Schemes}\label{sec:coherent-and-quot-reps}
We now recall the definition of derived coherent sheaves relative to a morphism of stacks. The following definitions can be found (with homological rather than cohomological conventions) at the beginning of \cite[\S2]{PS}; we reproduce them here for the reader's convenience. Recall that for $A\in \mathsf{CAlg}$, an $A$-module $M$ is said to be \textbf{almost perfect} if it is eventually connective and all truncations $\tau^{\ge i}M$ are compact in $A\Mod^{\ge i}$. We denote the full subcategory of $A\Mod$ of almost perfect modules by $\mathsf{APerf}(\Spec A)$. For any prestack $X$, we can define its \textbf{category of almost perfect complexes} 
 \begin{align*}
 	\mathsf{APerf}(X)\coloneqq \lim_{S\to X}\mathsf{APerf}(S).
 \end{align*} 
 Given a morphism $f:X\to S$ in $\mathsf{Stk}$, a quasi-coherent sheaf $\mathscr{F}\in \mathsf{QCoh}(X)$ has \textbf{Tor amplitude $ {\ge - n}$ relative to $ {S}$} if for every $\mathscr{G}\in \mathsf{QCoh}(S)^\heartsuit$,
 \begin{align*}
 	\mathrm H^{-i}(\mathscr{F}\otimes f^*\mathscr{G})=0
 \end{align*} 
 for $i\notin[0,n]$. The full subcategory of $\mathsf{APerf}(X)$ spanned by sheaves of Tor amplitude $\ge -n$ relative to $S$ is denoted $\mathsf{APerf}^{\ge -n}_S(X)$. Define
 \begin{align*}
 	\mathsf{Coh}_S(X)\coloneqq \mathsf{APerf}^{\ge 0}_S(X)
 \end{align*} 
 as the category of \textbf{flat families of coherent sheaves on $ {X}$ relative to $ {S}$}. By Lemma 2.4 of \textit{op. cit.}, the assignment
 \begin{align*}
	 \mathsf{Aff}^{\text{op}}\to &\mathsf{Spc}\\
	 S\mapsto &\ywidetilde{\mathsf{Coh}_S(X\times S)}
 \end{align*} 
defines a stack $\Coh(X)$ called the \textbf{derived stack of coherent sheaves on $ {X}$}. Now, if $X\to S$ is a map of schemes locally almost of finite type (laft) and $\mathscr{F}\in \mathsf{APerf}(X)$, we say that $\mathscr{F}$ has \textbf{proper support relative to $ {S}$} if there is a closed subscheme $Z\hookrightarrow X$ such that $Z\to S$ is proper and $\mathscr{F}|_{X\setminus Z}\cong 0$. Denote the category of almost perfect complexes with proper support by $\mathsf{APerf}_{\text{prop}}(X)$. For a quasi-projective classical scheme $X$, define the \textbf{stack of coherent sheaves with proper support} as
\begin{align*}
	\Coh_{\text{prop}}(X)\coloneqq \APerf_{\text{prop}}(X)\times _{\APerf(X)}\Coh(X).
\end{align*} 
\begin{rmk}
	In \cite{PS}, this stack is defined as
	\begin{align*}
		\Perf_{\text{prop}}(X)\times _{\Perf(X)}\Coh(X),
	\end{align*} 
	but the map $\Coh(X)\to \Perf(X)$ does not exist unless $X$ is smooth (see \cite[Lemma 2.18 and Corollary 2.19]{PS}). The first author of \textit{op. cit.} has informed me in private correspondence that this fiber product was intended to be as in our definition above.
\end{rmk} 
\begin{ex}
	Let $X=\Spec k[x]/(x^2)$, $k \in \mathsf{QCoh}(X)$, and $\pt = \Spec k$. Then $k$ is almost perfect, as it has bounded, coherent cohomologies. It is also of Tor amplitude 0 relative to $\pt$, since if $V\in \mathsf{Vect}^\heartsuit$, then
	\begin{align*}
		k\otimes_{k[x]/(x^2)}(k[x]/(x^2)\otimes V) = V
	\end{align*} 
	is Tor amplitude 0. But it is not perfect, since it has infinite projective dimension as a $k[x]/(x^2)$-module. Thus, it is not an object of
	\begin{align*}
		\Perf_{\text{prop}}(X)(\pt)\times _{\APerf(X)(\pt)}\Coh(X)(\pt).
	\end{align*} 
	Since our eventual goal is a definition of the representation moduli and, for an algebra $A$, $k\in \Rep_1(A)(k)$, we see that the definition of the stack of coherent sheaves with proper support given above is more suitable for a definition of representation varieties.
\end{ex}
\begin{prop}\label{prop:coherent-sheaves-with-proper-support-is-1-artin-laft-affine-diagonal}
	Let $X$ be a classical quasi-projective scheme. Then the stack
	\begin{align*}
		\Coh_{\text{prop}}(X)
	\end{align*} 
	is a 1-Artin stack locally almost of finite type with affine diagonal.
\end{prop} 
\begin{proof}
	The proof of Proposition 2.29 of \cite{PS} applies verbatim, as it does not use the fact that the moduli parametrizes perfect, rather than almost perfect, complexes. Similarly, the morphism $\Coh(X)\to \APerf(X)$ is an infinitesimally cohesive and nilcomplete morphism of stacks which are themselves infinitesimally cohesive and nilcomplete. Since these properties are stable under base change (cf. the proof of Corollary 2.14 of \textit{op. cit.}), it follows that $\Coh_{\text{prop}}(X)\to \APerf_{\text{prop}}(X)$ is infinitesimally cohesive and nilcomplete. It remains to check that $\APerf_{\text{prop}}(X)\to \APerf(X)$ is infinitesimally cohesive and nilcomplete. Once we have this, we can use \cite[Theorem 18.1.0.2]{sag} to conclude that $\Coh_{\text{prop}}(X)$ is a 1-Artin stack locally almost of finite type. 

	We now need the following result, which says that proper support is detected on the classical truncation (in fact, on the underlying reduced, classical scheme).
	\begin{lem}\label{lem:proper-support-is-detected-on-the-classical-truncation}
		If $S$ is an affine scheme and $i:X\times S_{\text{red}}\to X\times S$ is the inclusion, then $\mathscr{F}\in \mathsf{APerf}(X \times S)$ has proper support relative to $S$ if and only if $i^*\mathscr{F}$ has proper support relative to $S_{\text{red}}$.
	\end{lem}  
	\begin{proof}[Proof of Lemma]
	 The forward implication is immediate by base change, while the reverse implication can be seen as follows. Suppose $i^*\mathscr{F}$ is supported on a closed subscheme $Z \hookrightarrow X\times S_{\text{red}}$ such that $Z \to S_{\text{red}}$ is proper. Then the composition $Z \to S_{\text{red}}\to S$ is also proper, so we see that $Z $ is a closed subscheme of $X\times S$ proper over $S$. Denoting $U\coloneqq (X\times S)\setminus Z $, we see that $i^*(\mathscr{F}|_U)\cong 0$. It is now sufficient to recall that pullbacks to geometric fibers are jointly conservative on almost perfect complexes to conclude. 
\end{proof}
	Now, to show that $\APerf_{\text{prop}}(X)\to \APerf(X)$ is nilcomplete, we would like to show that given an affine $S$ and $S_n=\tau_{\ge n}S$ its truncations,
	\begin{align*}
		\APerf_{\text{prop}}(X)(S)\to \APerf(X)(S)\times _{\lim \APerf(X)(S_n)}\lim \APerf_{\text{prop}}(X)(S_n)
	\end{align*} 
	is an equivalence. If $\mathscr{F}\in \mathsf{APerf}(X\times S)$ and every $\mathscr{F}|_{X\times S_n}$ has proper support, then its restriction to $S_{\text{red}}$ has proper support. Thus, $\mathscr{F}$ has proper support by the above Lemma. This proves essential surjectivity, while fully faithfulness follows from the fact that the map $\APerf_{\text{prop}}(X)\to \APerf(X)$ is an inclusion.

	To show infinitesimal cohesiveness, we follow the proof of \cite[Lemma 2.13]{PS}. Let $S=\Spec A$ be an affine scheme, $M\in \mathsf{QCoh}(S)^{\ge 1}$, $S[M]\coloneqq \Spec(A\oplus M)$, and $d:S[M]\to S$ be a derivation, and consider the pushout diagram
	\[
	\begin{tikzcd}
		S[M]\arrow[r,"d"]\arrow[d,"d_0"]&S\arrow[d,"f_0"]\\
		S\arrow[r,"f"]&S_d[M[-1]]
	\end{tikzcd},
	\] 
	where $d_0$ is the 0 derivation. Notice that all of the schemes in the diagram above have the same classical reduction $S_{\text{red}}$. Let $\mathscr{F}\in \mathsf{APerf}(X\times S_d[M[-1]])$, and suppose the restrictions of $\mathscr{F}$ along $f,f_0$ and $f_0d\cong fd_0$ have proper support. Since the subsequent restrictions to $S_{\text{red}}$ give the same almost perfect complex with proper support, we see that $\mathscr{F}$ has proper support over $S_d[M[-1]]$ by Lemma \ref{lem:proper-support-is-detected-on-the-classical-truncation}. This proves essential surjectivity of the relevant map; fully faithfulness follows because $\APerf_{\text{prop}}(X)\to \APerf(X)$ is an inclusion.

	Finally, the stack $\Coh_{\text{prop}}(X)$ has affine diagonal by \cite[Chapter 2 Corollary 3.2.8]{GR}, as its classical truncation is the stack of coherent complexes with proper support, which has affine diagonal by \cite[Tag 0DLY]{stacks}. 
\end{proof} 
 
\begin{thm}\label{thm:coherent-reps}
	Let $\mathcal{C}=\mathsf{QCoh}(X)$ for a classical quasi-projective scheme $X$. Then $\mathscr{Rep}_n(\mathcal{C})$ has a substack
	\begin{align*}
		\mathscr{Rep}_n^{\text{coh}}(X)\coloneqq \Coh_{\text{prop}}(X)\times _{\Vect}\pt\subset \mathscr{Rep}_n(\mathcal{C}),
	\end{align*} 
	called the \textbf{coherent representation moduli}, which is a 1-Artin stack locally almost of finite type (laft). Moreover, this stack has affine diagonal.
\end{thm} 
\begin{proof}
	We will show that $\mathscr{Rep}_n^{\text{coh}}(X)$ is an open substack of a stack $\mathscr{V}$ which is affine and almost of finite type over $\mathsf{Coh}_{\text{prop}}(X)$. This proves that $\mathscr{Rep}_n^{\text{coh}}(X)$ is quasi-affine and almost of finite type over $\mathsf{Coh}_{\text{prop}}(X)$. Since $\Coh_{\text{prop}}(X)$ is 1-Artin and laft by Proposition \ref{prop:coherent-sheaves-with-proper-support-is-1-artin-laft-affine-diagonal}, this proves that $\mathscr{Rep}_n^{\text{coh}}(X)$ is also 1-Artin and laft.

	We first show that if $\mathscr{E}\in \mathsf{Coh}_\text{prop}(X)(S)$ is an $S$-point, then $\mathscr{V}_{\mathscr{E}}=p_{S*}\mathscr{E}$, where $p_S:X\times S\to S$, is perfect and its construction commutes with arbitrary base change. Indeed, to see that we have base change, it is sufficient to observe that $p_S$ is schematic quasi-compact, hence satisfies the conditions of \cite[Chapter 3 Section 2.2]{GR}. By \cite[Proposition 2.28]{PS}, $\mathscr{V}_{\mathscr{E}}$ is almost perfect. Let $M\in \mathsf{QCoh}(S)^\heartsuit$. Then $\mathscr{V}_{\mathscr{E}}\otimes_{\mathscr{O}_S}M\cong p_{S*}(\mathscr{E}\otimes _{\mathscr{O}_{X\times S}}p_{S^*}M)$ by the projection formula (\cite[Chapter 3 Lemma 3.2.4]{GR}). By definition of $\mathscr{E}$, $\mathscr{E}\otimes_{\mathscr{O}_{X\times S}} p_{S^*}M$ is Tor amplitude 0. Since $X$ is quasi-compact, there is a finite Zariski cover; thus, for $\mathscr{F}\in \mathsf{QCoh}(X\times S)^\heartsuit$, $p_*\mathscr{F}$ is computed by a finite \v{C}ech complex where the terms are $t$-exact pushforwards of restrictions of $\mathscr{F}$. Thus, $p_*\mathscr{F}$ is of finite Tor amplitude; moreover, it is coconnective. It follows that $\mathscr{V}_{\mathscr{E}}$ is almost perfect of Tor amplitude $[0,a]$, hence perfect \cite[Proposition 7.2.4.23]{HA}.

	Because of the base change compatibility of the definition of $\mathscr{V}_{\mathscr{E}}$, the assignments $\mathscr{E}\mapsto p_*\mathscr{E}$ glue to give a perfect complex $\mathscr{V}\in\mathsf{Perf}(\Coh_{\text{prop}}(X))$. Then $(\mathscr{V}^n)^\vee$ is a connective perfect complex; hence, it defines a stack
\begin{align*}
	\mathbb{V}\coloneqq \Spec_{\Coh_{\text{prop}}(X)}\Sym((\mathscr{V}^n)^\vee)\xrightarrow{p} \Coh_{\text{prop}}(X)
\end{align*} 
which is affine almost of finite type over $\Coh_{\text{prop}}(X)$. Notice that $S$-points of $\mathbb{V}$ are precisely given by $S$-points $f:S\to \Coh_{\text{prop}}(X)$ together with a map $\mathscr{O}^n_S\to f^*\mathscr{V} = p_{S*}\mathscr{E}$ by the base change compatibility of $\mathscr{V}$. We thus see that $S$-points of $\mathbb V$ are pairs of an almost perfect complex $\mathscr{E}$ on $X\times S$ and a map $\mathscr{O}^n_S\to p_{S*}\mathscr{E}$.

It remains to show that $\mathscr{Rep}_n^{\text{coh}}(X)$ is an open subfunctor of this stack. Let $\mathscr{O}^n_{\mathbb{V} }\to p^*\mathscr{V}$ be the universal map, and observe that its cofiber is a perfect complex. Clearly, the locus where that cofiber vanishes is open; moreover, that locus is precisely the locus where the map
\begin{align*}
	\mathscr{O}^n_{\mathbb{V} }\to p^*\mathscr{V}
\end{align*} 
is an isomorphism.

Now, we will show that $\mathscr{Rep}_n^{\text{coh}}(X)$ has affine diagonal. Note that $\Coh_{\text{prop}}(X)$ has affine diagonal by Proposition \ref{prop:coherent-sheaves-with-proper-support-is-1-artin-laft-affine-diagonal}. Since $\mathbb V\to \Coh_{\text{prop}}(X)$ is affine, it is separated, so the diagonal 
\begin{align*}
	\Delta_{\mathbb{V} }:\mathbb{V} \to \mathbb{V} \times _{\Coh_{\text{prop}}(X)}\mathbb{V} 
\end{align*} 
is also affine. We see that the diagonal of $\mathscr{Rep}_n^{\text{coh}}(X)$ factors as
\begin{align*}
	\Delta_{\mathscr{Rep}_n^{\text{coh}}(X)}:\mathscr{Rep}_n^{\text{coh}}(X)\to \mathscr{Rep}_n^{\text{coh}}(X)\times _{\Coh_{\text{prop}}(X)}\mathscr{Rep}_n^{\text{coh}}(X)\to \mathscr{Rep}_n^{\text{coh}}(X)\times \mathscr{Rep}_n^{\text{coh}}(X),
\end{align*} 
where the first map is affine because the map $\mathscr{Rep}_n^{\text{coh}}(X)\to \Coh_{\text{prop}}(X)$ is quasi-affine, hence has affine relative diagonal (it is a base change of $\Delta_\mathbb{V} $), and the second is affine because it is the base change of the affine diagonal of $\Coh_{\text{prop}}(X)$.
\end{proof}
For a classical quasi-projective scheme $X$, we have the substack
\begin{align*}
\Coh_{\text{prop}}^{\le d}(X)\subset \Coh_{\text{prop}}(X)
\end{align*} 
of sheaves whose proper support is dimension $\le d$. The stack $\Coh_{\text{prop}}^{\le 0}(X)$ has a decomposition
\begin{align*}
	\Coh_{\text{prop}}^{\le 0}(X)=\coprod_{m\in \mathbb{N}_0}\Coh_0^m(X)
\end{align*} 
by substacks of sheaves with Hilbert polynomial identically equal to $m$, which follows from the decomposition of the underlying classical stack. In particular, we have a stack
\begin{align*}
	\Coh^n_0(X)
\end{align*} 
of length $n$ sheaves on $X$.
\begin{cor}\label{cor:quot-is-a-scheme}
Let $X$ be a quasi-projective classical scheme. Then the fiber
\begin{align*}
	\Quot^{n,\text{fr}}_{\mathscr{O}^n_{X}}(X)\coloneqq \Coh^n_0(X)\times _{\Vect}\pt\subset \mathscr{Rep}_n(X)
\end{align*} 
is a quasi-compact scheme almost of finite type with affine diagonal.
\end{cor} 
\begin{proof}
	Since
	\begin{align*}
		Y=\Quot^{n,\text{fr}}_{\mathscr{O}_X^n}(X)\subset \mathscr{Rep}_n^{\text{coh}}(X)
	\end{align*} 
	has classical truncation a scheme of finite type by the proof of \cite[Proposition 4.4.1]{KV}, it is a quasi-compact scheme almost of finite type. This immediately implies that $Y$ is Deligne-Mumford since its classical truncation is \cite[Definition 8.5.7(iii)]{khan-underlying-classical-is-dm-implies-dm}. One can then use \cite[Corollary 1.6.7.4]{sag} to conclude that $Y$ is in fact a scheme. 

	Further, the scheme $Y$ has affine diagonal because it is a substack of $\mathscr{Rep}_n^{\text{coh}}(X)$, which has affine diagonal by Theorem \ref{thm:coherent-reps}.
\end{proof} 
\begin{rmk}
	An $S$-point of $\Quot^{n,\text{fr}}_{\mathscr{O}_X^n}(X)$ is thus a family $\mathscr{E}\in \mathsf{Coh}^n_0(X)(S)$ together with an isomorphism
	\begin{align*}
		\mathscr{O}^n_S\xrightarrow{\cong}p_{S*}\mathscr{E}.
	\end{align*} 
	By adjunction, this isomorphism determines a quotient map $q: \mathscr{O}^n_{X\times S}\to \mathscr{E}$ which is surjective on $\mathrm H^0$. Hence, the classical truncation of this derived scheme is the open locus of the classical Quot scheme of length $n$ quotients whose adjoint $\mathscr{O}^n_S\to p_{S^*}\mathscr{E}$ is an isomorphism. We call this the framed locus.
\end{rmk} 

Even for affine $X=\Spec A$, we don't necessarily have that coherent representations match all representations. Indeed, suppose $A $ is the free commutative algebra on infinitely many generators. Then $k$ is not finitely presented as an $A$-module. By \cite[Proposition 7.2.4.17]{HA}, this implies that it is not almost perfect as an $A$-module, hence not coherent in the above sense. However, we do have the following.
\begin{prop}\label{prop:quot-scheme-is-reps}
 	Let $A$ be a classical finitely presented commutative algebra. Then 
	\begin{align*}
		\Rep_n(A)\cong \Quot^{n,\text{fr}}_{A^n}(\Spec A).
	\end{align*} 
 \end{prop} 
 \begin{proof}
	 Note that by Theorem \ref{thm:rep-on-affines}, we have an inclusion
	 \begin{align*}
		 \Quot^{n,\text{fr}}_{A^n}(\Spec A) \hookrightarrow \Rep_n(A),
	 \end{align*} 
	 and it is sufficient to show that this map is surjective. Let $\Spec B$ be an affine scheme. Any almost perfect $(A\otimes B)$-module $M$ such that $M\cong B^n$ as $B$-modules has support which is finite over $\Spec B$; hence, the coherent sheaf on $\Spec(A\otimes B)$ which it defines has proper support relative to $\Spec B$. Since the pushforward of the sheaf corresponding to $M$ to $\Spec B$ is locally free of rank $n$ and its support is finite over $\Spec B$, it is a flat family of length $n$ sheaves; hence, it defines a point of $\mathsf{Coh}_0^n(\Spec A)$. Clearly, $M$ is Tor amplitude 0 over $B$, as, for $N\in B\Mod^\heartsuit$,
	 \begin{align*}
	 	M\otimes _{A\otimes B}(A\otimes B \otimes _B N)\cong M\otimes _B N \cong B^n\otimes _B N\cong N^n.
	 \end{align*}  Thus, it is sufficient to show that any $(A\otimes B)$-module $M$ with an isomorphism $M\cong B^n$ is almost perfect.

	 We will first prove the following.
	 \begin{lem}
	 	Let $R\to S$ be a map of connective algebras such that $S$ is almost perfect as an $(S\otimes _{R}S)$-module. If $M$ is an $S$-module which is almost perfect over $R$, then it is almost perfect over $S$.
	 \end{lem} 
	 \begin{proof}[Proof of Lemma]
	 	 Without loss of generality, we can assume that $M$ is connective; indeed, since $M$ is almost perfect over $R$, it is bounded above. Let $T=S\otimes _RS$, and consider the functor
		 \begin{align*}
		 	-\otimes _SM:T\Mod\to S\Mod,
		 \end{align*} 
		 where $N\otimes _SM$ has an $S$-module structure via the left $S$-action on the first factor. Since $S$ is almost perfect over $T$, we can, by \cite[Proposition 7.2.4.11]{HA}, choose for each $m$ a perfect $T$-module $P_m$ equipped with a map $P_m\to S$ that has $m$-connective fiber. Applying $-\otimes _SM$ gives a map $P_m\otimes _SM\to M$ which has $m$-connective fiber. Since $P_m\in \mathsf{Perf}(T)$, $P_m\otimes_S M$ is in the thick subcategory of $S$-mod generated by $S\otimes _R M$. Since $S\otimes _RM$ is almost perfect over $S$ by base change of the almost perfect $R$-module $M$ and almost perfect modules are closed under finite colimits and retracts by \textit{loc. cit.}, $P_m\otimes _SM$ is almost perfect over $S$. Now, if $M$ is bounded above and for every $m$ there is an almost perfect $Q_m\to M$ whose fiber is $m$-connective, then $M$ is almost perfect by \textit{loc. cit.}
	 \end{proof}
	 We would like to apply the Lemma to our situation, where $R=B$, $S=A\otimes B$, which will conclude the proof that $M$ is almost perfect over $A\otimes B$. Note that $R$ and $S$ are connective because $\Spec B$ and $\Spec (A\otimes B)$ are affine schemes, and $M$ is almost perfect over $B$. Since $A$ is finitely presented over $k$, it is almost perfect as a module over $A\otimes A$. Indeed, $A\otimes A$ is left coherent (because it's noetherian); hence, by \cite[Proposition 7.2.4.17]{HA}, $A$ is almost perfect over $A\otimes A$. Indeed, it is connective, and the multiplication
	 \begin{align*}
	 	A\otimes A\to A
	 \end{align*} 
	 has finitely generated image. Thus, base change along a connective algebra $B$ gives
	 \begin{align*}
	 	A\otimes B\cong A\otimes _{A\otimes A}(A\otimes A\otimes B),
	 \end{align*} 
	 so $A\otimes B$ is almost perfect over $A\otimes A\otimes B=S\otimes _RS$.

\end{proof}

 \begin{cor}\label{cor:coherent-reps-match-reps}
	If $X=\Spec A$ is affine of finite presentation, then
	\begin{align*}
		\mathscr{Rep}_n^{\text{coh}}(X)\cong\Rep_n(A).
	\end{align*} 
\end{cor} 

\begin{rmk}
	We have thus shown that in the case of $X=\Spec A$ of finite presentation, coherent moduli, representation moduli, derived representation varieties, and the derived framed locus $\Quot ^{n,\text{fr}}_{\mathscr{O}_X^n}(X)$ all agree. This completes the proof of Theorem \ref{thm:intro-affine-reps-all-agree} from the introduction.
\end{rmk} 

\subsection{The Universal Representation}\label{sec:universal-sheaves}
From now on, let $X$ be a classical quasi-projective scheme. Associated to the identity map
\begin{align*}
	\Coh^n_0(X)\to \Coh^n_0(X),
\end{align*} 
there is a universal sheaf $\mathscr{U}$ on
\begin{align*}
	\Coh^n_0(X)\times X
\end{align*} 
defined as follows. We can extend the moduli functor $\Coh^n_0(X)$ from affine schemes to stacks by right Kan extension along the inclusion $\mathsf{Aff}\to \mathsf{Stk}$. Explicitly, if $ Z$ is a stack, then we define
\begin{align*}
	\Coh^n_0(X)(Z)\coloneqq \lim_{S\to Z}\Coh^n_0(X)(S),
\end{align*} 
where the limit is taken over all affine schemes $S$ mapping to $Z$. Thus, there are equivalences of spaces
\begin{align*}
	\mathsf{Map}(\Coh^n_0(X),\Coh^n_0(X))\cong&\lim_{S\to \Coh^n_0(X)}\Coh^n_0(X)(S).
\end{align*} 
In particular, we see that there is an object in
\begin{align*}
	\lim_{S\to \Coh^n_0(X)}\Coh^n_0(X)(S)
\end{align*} 
which corresponds to the identity morphism $\Coh^n_0(X)\to \Coh^n_0(X)$. Since $\Coh^n_0(X)(S)\subset \Coh_{\text{prop}}(X)(S)\subset \Coh(X)(S)\subset \mathsf{APerf}(X\times S)$, descent for almost perfect complexes implies that
\begin{align*}
	\mathsf{APerf}(\Coh^n_0(X) \times X)\cong \lim_{S\to \Coh^n_0(X)}\mathsf{APerf}(S\times X)\supset \lim_{S\to \Coh^n_0(X)}\Coh^n_0(X)(S).
\end{align*} 
This defines an almost perfect complex $\mathscr{U}\in \mathsf{APerf}(\Coh^n_0(X)\times X)$. Moreover, if $f:S\to \Coh^n_0(X)$ is an $S$-point, then the pullback of $\mathscr{U}$ along $f\times \id_X:S\times X\to \Coh^n_0(X)\times X$ gives a length $n$ torsion sheaf on $S\times X$ with proper support relative to $S$.

Informally, this sheaf can be understood as follows. Given an $S$-point 
\begin{align*}
	(f,x):S\to \Coh^n_0(X)\times X
\end{align*} 
we obtain a sheaf on $S\times X$ via the map $f$ together with a map $x:S\to X$. By pulling back this sheaf along the graph of $x$ in $S\times X$, we obtain a sheaf on $S$, and this is precisely the pullback of the sheaf $\mathscr{U}$ to $S$ along $(f,x)$.

Denote $Y=\Quot^{n,\text{fr}}_{\mathscr{O}_X^n}(X)$. Pulling back $\mathscr{U}$ along
\begin{align*}
	\pi \times \id_X:Y\times X\to \Coh^n_0(X) \times X,
\end{align*} 
where $\pi:Y\to \Coh^n_0(X)$ is the projection, gives a universal sheaf on $Y\times X$,
\begin{align*}
	\mathscr{R}\coloneqq (\pi\times \id_X)^*\mathscr{U}.
\end{align*} 
Note that $\mathscr{R}$ is an almost perfect complex.
\begin{lem}\label{lem:R-has-proper-support}
	The sheaf $\mathscr{R}$ has proper support relative to $Y= \Quot^{n,\text{fr}}_{\mathscr{O}_X^n}(X)$.
\end{lem} 
\begin{proof}
	Let $Z\subset (Y\times X)^{\text{cl}}\subset Y\times X$ be the subscheme of $Y\times X$ defined by equipping the classical support of $\mathscr{R}|_{(Y\times X)^{\text{cl}}}$ with the reduced induced subscheme structure. Note that $Z$ is closed because $\mathscr{R}$ is almost perfect and flat relative to $Y$; hence, its restriction to the classical truncation is a coherent sheaf in degree 0 and therefore has closed support. Since an almost perfect complex is 0 if it is 0 on the classical truncation,\footnote{If $B$ is a connective commutative algebra and $M$ is a bounded above $B$-module with highest nonzero cohomology module in degree $m$, then $\mathrm H^m(\mathrm H^0(B)\otimes _BM)\cong \mathrm H^m(M)$, and this is nonzero if and only if $\mathrm H^m(M)$ is. Thus, we see that restriction to the classical truncation is indeed conservative on bounded above complexes.} we see that
\begin{align*}
	\mathscr{R}|_{(Y\times X)\setminus Z}\cong 0.
\end{align*} 
It remains to check that $Z\to Y$ is proper, which is equivalent to checking that $Z\to Y^{\text{cl}}$ is proper. Since properness is affine-local on the target, it is sufficient to check that for $S$ an affine open of $Y^{\text{cl}}$, $\mathscr{R}|_{S \times X}$ has proper support relative to $S$. We observe that $Z_S=S\times_{Y^{\text{cl}}} Z$ is the classical support of $\mathscr{R}|_{S\times X}$ with its reduced induced subscheme structure (by \cite[Tag 056J]{stacks} applied to the cohomology sheaves of $\mathscr{R}|_{S\times X}$). Since $\mathscr{R}|_{S\times X}$ has proper support relative to $S$ via the composition $S\to Y\to \Coh^n_0(X)\subset \Coh_{\text{prop}}(X)$, there is a closed subscheme $Z'\subset S\times X$ proper over $S$ with $\mathscr{R}|_{(S\times X)\setminus Z'}\cong 0$; moreover, $Z_S$ is a closed subscheme of $Z '$ by definition. Thus, we see that $Z_S\to S$ is proper, which proves the claim.
\end{proof} 
\begin{cor}\label{cor:R-has-finite-support}
	There is a closed subscheme $Z\subset Y\times X$ such that $\mathscr{R}|_{(Y\times X)\setminus Z}\cong 0$ and the composition $Z\to Y\times X\to Y$ is finite.	
\end{cor} 
\begin{proof}
	The scheme $Z$ in the proof of Lemma \ref{lem:R-has-proper-support} has finite fibers over $Y$; hence, $Z\to Y$ is quasi-finite. Using Lemma \ref{lem:R-has-proper-support}, we see that $Z\to Y$ is proper and quasi-finite, hence finite.
\end{proof} 
\begin{prop}\label{prop:R-is-free-of-rank-n}
 Denote $Y= \Quot^{n,\text{fr}}_{\mathscr{O}_X^n}(X)$, and let $p_Y:Y\times X\to Y$ be the projection. We then have
 \begin{align*}
 	p_{Y*}\mathscr{R}\cong \mathscr{O}^n_Y.
 \end{align*} 
\end{prop} 
\begin{rmk}
This proposition shows that $\mathscr{R}$ is the analog in the non-affine setting of the $(A,\mathbb{L}(A)_n)$-bimodule $(\mathbb{L}(A)_n)^n$ of Lemma \ref{lem:A-module-structure}.
\end{rmk} 
\begin{proof}
	By considering the Kan-extended functor of points of $Y$, which takes values on stacks, we see that the map $Y\to Y$ is a tautological point of $Y(Y) $; hence, there is a universal equivalence $\mathscr{O}^n_Y\cong p_{Y*}\mathscr{R}$. Indeed, given $y:S\to Y$, the composition
	\begin{align*}
		S\xrightarrow{y} Y\xrightarrow{\pi} \Coh^n_0(X)
	\end{align*} 
	thus classifies a sheaf on $S \times X$, and this sheaf corresponds to the sheaf on $S\times X$
	\begin{align*}
		 \mathscr{R} | _{S\times X}=(y\times \id_X)^*\mathscr{R}=((\pi\circ y)\times \id_X)^*\mathscr{U}
	\end{align*} 
	by universality of $\mathscr{U}$. Further, under pushforward along $p_S:S\times X\to S$, we have an isomorphism
	\begin{align*}
		p_{S*}\mathscr{R}|_{S\times X}\cong \mathscr{O}^n_S
	\end{align*} 
	by construction of $Y$. Now, consider the following cartesian square
	\[
	\begin{tikzcd}
		S\times X\arrow[r,"y\times \id_X"]\arrow[d,"p_S"]&Y\times X\arrow[d,"p_Y"]\\
		S\arrow[r,"y"]&Y
	\end{tikzcd}.
	\] 
Base change along this square gives an isomorphism
\begin{align*}
	p_{S*}(y\times \id_X)^*\cong y^*p_{Y*},
\end{align*} 
so that
\begin{align*}
p_{S*}\mathscr{R}|_{S\times X}\cong (p_{Y*}\mathscr{R})|_{S}\cong \mathscr{O}^n_S.
\end{align*} 
This thus identifies
\begin{align*}
	p_{Y*}\mathscr{R}\cong \mathscr{O}^n_Y.
\end{align*} 
\end{proof}

\subsection{The Universal Representation Defines a Domain Wall}
By Corollary \ref{cor:quot-is-a-scheme}, $Y= \Quot^{n,\text{fr}}_{\mathscr{O}_X^n}(X)$ is a quasi-compact scheme almost of finite type with affine diagonal; hence, it is perfect in the sense of \cite{BFN}. Since $X$ and $Y$ are perfect, we have equivalences
\begin{align*}
	\mathsf{QCoh}(Y\times X)\cong \mathsf{QCoh}(Y)\otimes \mathsf{QCoh}(X)\cong \mathsf{Map}_{\mathsf{Pr}^{\text{L}}}(\mathsf{QCoh}(X),\mathsf{QCoh}(Y)),
\end{align*} 
by e.g. \cite[Proposition 3.3.4.2]{GR}. Thus, the sheaf $\mathscr{R}$ defines a left adjoint functor
\begin{align*}
	\Phi_{\mathscr{R}}:\mathsf{QCoh}(X)\to& \mathsf{QCoh}(Y)\\
	\mathscr{F}\mapsto &p_{Y*}\left( p_X^*\mathscr{F}\otimes \mathscr{R} \right) ,
\end{align*} 
where $p_{Y}$ and $p_X$ are the corresponding projections from $Y\times X$. This functor abstractly has a right adjoint, but this right adjoint may not necessarily preserve colimits and hence may not be a right adjoint \textit{in the category} $\mathsf{Pr}^{\text{L}}$.
\begin{prop}\label{prop:R-is-a-domain-wall}
	Let $X$ be a quasi-projective classical scheme. Then the right adjoint of $\Phi_{\mathscr{R}}$ preserves colimits. 
\end{prop} 
\begin{proof}
	It is sufficient to show that $\Phi_{\mathscr{R}}$ preserves compact objects. Since $X$ and $Y$ are perfect, so is the map $p_Y:Y\times X\to Y$ by \cite[Corollary 3.23]{BFN}. In particular, $p_Y$ satisfies the projection formula by \cite[Proposition 3.10]{BFN}. Further, the compact objects of $\mathsf{QCoh}(Z)$ for a perfect stack $Z$ with affine diagonal are precisely the perfect complexes $\mathsf{Perf}(Z)$ (\cite[Proposition 3.9]{BFN}). Thus, we need to show that given a perfect complex $\mathscr{P}$ on $X$, $\Phi_{\mathscr{R}}(\mathscr{P})=p_{Y*}(p_X^*\mathscr{P}\otimes \mathscr{R})$ is also perfect. Denote $\mathscr{E}\coloneqq p_X^*\mathscr{P}\otimes \mathscr{R}$. Since by Lemma \ref{lem:R-has-proper-support} $\mathscr{R}$ has proper support relative to $Y$, $\mathscr{E}$ has proper support relative to $Y$ and is almost perfect; hence, $p_{Y*}\mathscr{E}$ is almost perfect by \cite[Proposition 2.28]{PS}. It thus suffices to check that
	\begin{align*}
		p_{Y*}\mathscr{E}\otimes \mathscr{F}
	\end{align*} 
	has uniformly bounded cohomology for any $\mathscr{F}\in \mathsf{QCoh}(Y)^\heartsuit$. It must be bounded from above. To see it is bounded below, write
	\begin{align*}
		p_{Y*}\mathscr{E}\otimes \mathscr{F}=p_{Y*}(\mathscr{E}\otimes p_Y^*\mathscr{F})
	\end{align*} 
by the projection formula. By assumption, $\mathscr{E}\otimes p_Y^*\mathscr{F}$ has cohomology sheaves in a finite range because $\mathscr{P}$ is perfect. Since $\mathscr{R}$ has support contained in $Z\subset  Y\times X$ where $Z\to Y$ is finite by Corollary \ref{cor:R-has-finite-support}, so does $\mathscr{E}\otimes p_{Y}^*\mathscr{F} $. Therefore, $p_{Y*}$ is $t$-exact on $\mathscr{E}\otimes p_Y^*\mathscr{F}$.
\end{proof} 

The above proposition shows that $\Phi_{\mathscr{R}}$ is actually a 1-morphism in $(\mathsf{Pr}^L)^{\text{cont}}$, which is the $(\infty,2)$-subcategory of $\mathsf{Pr}^L$ of dualizable objects and 1-morphisms whose adjoints are also morphisms. Thus, we are in the setting of \cite[\S3.4]{nonlinear-trace}, so there is a well-defined map
\begin{align*}
	\mathrm H\mathrm H_*(X)\to \mathrm H\mathrm H_*(Y)
\end{align*} 
induced by the 1-morphism $\Phi_{\mathscr{R}}$, where for a stack $Z$, $\mathrm H\mathrm H_*(Z)\coloneqq \mathrm H\mathrm H_*(\mathsf{QCoh}(Z))$. In other words, $\Phi_{\mathscr{R}}$ is a ``domain wall'' in the sense of \S\ref{sec:domain-walls}. Constant loops are a map $Y\to \mathscr{L}Y$, hence they induce a map
\begin{align*}
	\mathrm H\mathrm H_*(Y)=\mathscr{O}(\mathscr{L}Y)\to \mathscr{O}(Y).
\end{align*} 
By composing this with the induced map on Hochschild homologies above, we have thus obtained a map
\begin{align*}
	\mathrm H\mathrm H_*(X)\to \mathscr{O}(Y).
\end{align*} 
Moreover, this composition is equivariant with respect to the natural $S^1$-actions (because each of the composed maps is by \cite[Section 2.3]{HSS}), so it descends to a map
\begin{align*}
	\mathrm H\mathrm C_*(X)\to \mathscr{O}(Y).
\end{align*} 
The next theorem shows that in the case of $X=\Spec A$ of finite presentation, this is the trace map of \cite{BKR}.
\begin{thm}\label{thm:stack-domain-wall}
	Let $X=\Spec A$ be an affine scheme of finite presentation. Then $\mathscr{R}=(\mathbb{L}(A)_n)^n$ as an $(A\otimes \mathbb{L}(A)_n)$-module. Moreover, the functors it defines,
	\begin{align*}
	\Phi_{\mathscr{R}}:	\mathsf{QCoh}(X)\leftrightarrows \mathsf{QCoh}(\Quot^{n,\text{fr}}_{\mathscr{O}_X^n}(X)):\Phi_{\mathscr{R}}^R,
	\end{align*} 
	are precisely those defined by the domain wall $(\mathbb{L}(A)_n)^n$ in \S\ref{sec:trace-maps-I}.
\end{thm} 
\begin{proof}
	Consider $\Quot^{n,\text{fr}}_{\mathscr{O}_X^n}(X)\times X$, which is isomorphic to $\Rep_n(A)\times X$ by Theorem \ref{thm:rep-on-affines} and Proposition \ref{prop:quot-scheme-is-reps}. By Proposition \ref{prop:R-is-free-of-rank-n}, we see immediately that as an $\mathbb{L}(A)_n$-module, $\mathscr{R}$ is free of rank $n$. We would now like to show that the $A$-module structure is precisely the one arising from Lemma \ref{lem:A-module-structure}. For $S=\Spec B$, an $S$-point $s:S\to \Rep _n(A)$ is precisely a map $QA\otimes B\to \End_B((s\times \id_X)^*\mathscr{R})= M_n(B)$, where
	\begin{align*}
		(s\times \id_X)^*\mathscr{R}=A\otimes B\otimes _{\mathbb{L}(A)_n\otimes A}(\mathbb{L}(A)_n)^n=B^n.
	\end{align*} 
	We claim that the maps
	\begin{align*}
		QA\to M_n(B)
	\end{align*} 
	are natural in $B$; hence, they give a map of $\mathbb{L}(A)_n$-algebras
	\begin{align*}
		QA\otimes \mathbb{L}(A)_n\to M_n(\mathbb{L}(A)_n).
	\end{align*} 
	Indeed, this follows from naturality of the universal sheaf $\mathscr{R}$ with respect to pullback. The $A$-module structure is then induced under $A\Mod=QA\Mod$. We next claim that this module structure is the ``universal representation'' defined in the proof of Lemma \ref{lem:A-module-structure}. We observe that the $QA$-action on the restriction of $\mathscr{R}$ to $\Spec(B)$-points of $\Rep_n(A)$ is via $QA\to M_n(B)$ by construction, which is by definition the same action as that in Lemma \ref{lem:A-module-structure}.

	Lastly, we identify the functors defined by $\mathscr{R}$. The functor
	\begin{align*}
		\mathsf{QCoh}(X)\to \mathsf{QCoh}(\Rep_n(A))
	\end{align*} 
	sends an $A$-module $M$ to
	\begin{align*}
		\Phi_{\mathscr{R}}(M)=&M\otimes _A (A\otimes \mathbb{L}(A)_n)\otimes _{A\otimes \mathbb{L}(A)_n} \mathscr{R}\\
		=&M\otimes _A \mathscr{R}.
	\end{align*} 
	The functor in the reverse direction (which is the right adjoint) sends an $\mathbb{L}(A)_n$-module $N$ to
	\begin{align*}
		N\otimes _{\mathbb{L}(A)_n}\mathscr{R}^R=N\otimes _{\mathbb{L}(A)_n}\Hom_{\mathbb{L}(A)_n}(\mathscr{R},\mathbb{L}(A)_n)\cong N^n.
	\end{align*} 
	These are the functors identified in \S\ref{sec:domain-walls}.
\end{proof} 

\section{The Shifted Symplectic Structure}
\label{sec:symplectic-structure}
For $X$ a smooth classical quasi-projective scheme, we will show that there is a $\GL_n$ action on $Y=\Quot^{n,\text{fr}}_{\mathscr{O}_X^n}(X)$ such that the quotient stack $Y /\GL_n$ admits a natural shifted symplectic structure (see \cite{PTVV} for what this means). In the case of $X=\Spec A$ with $\dim X=2$, this structure generalizes the 0-shifted symplectic structure constructed on the $\GL_n$ character variety of a genus $1$ surface by Goldman in \cite{gold}.
\begin{prop}\label{prop:quot-scheme-is-torsor}
	The derived framed locus of the Quot scheme of points has a presentation
	\begin{align*}
		\Quot^{n,\text{fr}}_{\mathscr{O}_X^n}(X)\cong\Coh^n_0(X)\times _{B\GL_n}\pt,
	\end{align*} 
	where $\Coh^n_0(X)\to B\GL_n$ classifies the pushforward of the universal family $\mathscr{U}$ to $\Coh^n_0(X)$. In particular, the projection
	\begin{align*}
		\Quot^{n,\text{fr}}_{\mathscr{O}^n_X}(X)\to \Coh^n_0(X)
	\end{align*} 
	is a $\GL_n$-torsor.
\end{prop} 
\begin{proof}
	Let $p:\Coh^n_0(X)\times X\to \Coh^n_0(X)$ be the projection, and denote $\mathscr{V}\coloneqq p_*\mathscr{U}$, where $\mathscr{U}$ is the universal sheaf of Section \ref{sec:universal-sheaves}. Note that $p_*\mathscr{U}$ is perfect. In fact, we claim that $\mathscr{V}$ is locally free of rank $n$. It suffices to show that given an $S$-point $f:S\to \Coh^n_0(X)$, $f^*\mathscr{V}$ is locally free of rank $n$. The map $p$ is schematic quasi-compact; indeed, it is the base change of such a map. Thus, by \cite[Chapter 3 Proposition 2.2.2]{GR}, base change holds for the diagram
	\[
	\begin{tikzcd}
		S\times X\arrow[r,"f\times \id_X"]\arrow[d,"p_S"]&\Coh^n_0(X)\times X\arrow[d,"p"]\\
		S\arrow[r,"f"]&\Coh^n_0(X)
	\end{tikzcd}.
	\] 
This allows us to write
\begin{align*}
	f^*\mathscr{V}\cong p_{S*}(f\times \id_X)^*\mathscr{U}.
\end{align*} 
By construction of the universal family, $f^*\mathscr{V}$ is precisely the pushforward along $p_{S}$ of the sheaf on $S\times X$ classified by $f$, which we denote by $\mathscr{E}$. At any geometric point $s$ of $S$, $f^*\mathscr{V}\otimes k(s)\cong\Gamma(X_s,\mathscr{E}_s)\cong k(s)^n$ in degree 0, since $\mathscr{E}_s$ has length $n$. Thus, the bundle $\mathscr{V}$ defines a map
\begin{align*}
	\Coh^n_0(X)\to B\GL_n.
\end{align*} 
Since $B\GL_n\hookrightarrow \Vect$ is a monomorphism in $\mathsf{Stk}$ and $\mathscr{V}$ factors through it, we immediately have an identification
\begin{align*}
	\Coh^n_0(X)\times _{B\GL_n}\pt\cong\Coh^n_0(X)\times _{\Vect}\pt\eqcolon \Quot^{n,\text{fr}}_{\mathscr{O}_X^n}(X)
\end{align*} 
as desired.
\end{proof} 
The stack $\Coh^n_0(X)$ should be thought of as the analog of the ($\GL_n$) character stack when $X$ is not affine. Because
\begin{align*}
	\Coh^n_0(X)\subset \Coh_{\text{prop}}(X)
\end{align*} 
is an open and closed substack, its inclusion into $\Coh_{\text{prop}}(X)$ is formally \'etale. If $X$ is furthermore smooth, there is a map
\begin{align*}
	\Coh_{\text{prop}}(X)\to \Perf_{\text{prop}}(X) \cong \APerf_{\text{prop}}(X)\times _{\APerf(X)}\Perf(X),
\end{align*} 
which exists by \cite[Corollary 2.19]{PS}, and this map is formally \'etale because it is the base change of a formally \'etale map by \textit{loc. cit.} Thus, the composition
\begin{align*}
	\Coh^n_0(X)\to \Perf_{\text{prop}}(X)
\end{align*} 
is formally \'etale. By \cite{BD}, when $X$ is Gorenstein and Calabi-Yau of dimension $d$, the moduli stack of perfect complexes with proper support has a natural $(2-d)$-shifted symplectic structure. We have thus nearly obtained the following.
\begin{thm}\label{thm:symplectic-structure}
	Let $X$ be a classical, smooth, quasi-projective scheme of dimension $d$ equipped with a trivialization of its canonical bundle $\mathscr{O}_X\cong K_X$. Then the stack $\Coh^n_0(X) \cong \Quot^{n,\text{fr}}_{\mathscr{O}_X^n}(X) /\GL_n$ has a natural $(2-d)$-shifted symplectic structure.
\end{thm} 
\begin{proof}
	Denote the symplectic structure on $\Perf_{\text{prop}}(X)$ by $\omega$, and let
	\begin{align*}
	f:	\Coh^n_0(X)\to \Perf_{\text{prop}}(X)
	\end{align*} 
	be the map constructed above. Note that $f^*\omega$ is a closed 2-form of degree $2-d$ on $\Coh^n_0(X)$, and we need only show that it is nondegenerate. By nondegeneracy of $\omega$, the induced map
	\begin{align*}
		T_{\Perf_{\text{prop}}(X)}\to \mathbb{L}_{\Perf_{\text{prop}}(X)}[2-d]	
	\end{align*} 
	is an isomorphism. By the above, $f$ is formally \'etale. Since the relative cotangent complex of a formally \'etale map vanishes, we find that
	\begin{align*}
		T_{\Coh^n_0(X)}=f^*T_{\Perf_{\text{prop}}(X)}
	\end{align*} 
	and
	\begin{align*}
		\mathbb{L}_{\Coh^n_0(X)}=f^*\mathbb{L}_{\Perf_{\text{prop}}(X)}.
	\end{align*} 
Thus, the map
\begin{align*}
	T_{\Coh^n_0(X)}\to \mathbb{L}_{\Coh^n_0(X)}[2-d]
\end{align*} 
induced by $f^*\omega$ is an isomorphism as well.
\end{proof} 

\appendix
\section{Comparison with Existing Constructions}
\subsection{Quot Schemes}
Recall the construction of \cite[Definition 3.3]{quot}, denoted by $\Quot ^{\text{Adh}}_{\mathscr{O}_X^n}(X)$, where $X$ is a nonsingular quasi-projective scheme. For such $X$, define an open subfunctor of this scheme, $\Quot ^{\text{Adh,}n,\text{fr}}_{\mathscr{O}^n_X}(X)\subset \Quot ^{\text{Adh}}_{\mathscr{O}^n_X}(X)$, whose $S$-points are length $n$ quotients
\begin{align*}
	\mathscr{O}^n_{X\times S}\to \mathscr{E}
\end{align*} 
for which the adjoint map, under the adjunction 
	\begin{align*}
		\Hom_{X\times S}(\mathscr{O}^n_{X\times S},\mathscr{E})\cong\Hom_S(\mathscr{O}^n_S,p_{S*}\mathscr{E}),
	\end{align*} 
	where $p_S:X\times S\to S$ is the projection, is an isomorphism.
	\begin{prop}\label{prop:matchup-with-adhikari}
	Let $X$ be a smooth quasi-projective scheme. Then 
	\begin{align*}
		\Quot^{n,\text{fr}}_{\mathscr{O}_X^n}(X)\cong\Quot ^{\text{Adh},n,\text{fr}}_{\mathscr{O}^n_X}(X).
	\end{align*} 
\end{prop} 
\begin{proof}
	Let $\mathscr{E}\in \mathsf{Coh}^n_0(X)(S)$, for $S$ a test affine. By definition, an $S$-point of Adhikari's Quot scheme consists of a morphism
	\begin{align*}
		q: \mathscr{O}^n_{X\times S}\to \mathscr{E},
	\end{align*} 
	where $\mathrm H^0(q)$ is surjective and $\mathscr{E}$ is perfect, is flat over $S$, has proper support over $S$, and has geometric fibers of length $n$. We will construct a map
	\begin{align*}
		Y\coloneqq \Quot^{n,\text{fr}}_{\mathscr{O}_X^n}(X)\to \Quot ^{\text{Adh,}n,\text{fr}}_{\mathscr{O}^n_X}(X).
	\end{align*} 
	Recall the universal representation $\mathscr{R}$ of Section \ref{sec:universal-sheaves}, which is an almost perfect complex on $X\times Y$ with support proper over $Y$ (Lemma \ref{lem:R-has-proper-support}). It is equipped with a universal framing
	\begin{align*}
		\alpha_{\mathscr{R}}:\mathscr{O}^n_{Y}\xrightarrow{\cong}p_{Y*}\mathscr{R},
	\end{align*} 
	where $p_Y:X\times Y\to Y$. Consider the sheaf morphism $q_{\mathscr{R}}$ on $X\times Y$ given by the composition
	\begin{align*}
		q_{\mathscr{R}}:\mathscr{O}^n_{X\times Y}\xrightarrow{p_Y^*\alpha_{\mathscr{R}}}p_Y^*p_{Y*}\mathscr{R}\xrightarrow{\varepsilon_{\mathscr{R}}}\mathscr{R},
	\end{align*} 
	where $\varepsilon$ is the counit of the adjunction $(p_Y^*,p_{Y*})$. We next claim that $\mathrm H^0(q_{\mathscr{R}})$ is surjective. Denoting $C=\coker(\mathrm H^0(q_{\mathscr{R}}))$, it is sufficient to show that the restriction of $C$ to every geometric point is 0, by Nakayama and the fact that $C$ is of finite type. Thus, let $(x,y)$ be a geometric point of $X\times Y$, and set $K=k(y)$ and $\Omega=k(x,y)$. Let $\mathscr{R}_y$ denote the restriction of $\mathscr{R}$ to $X_y$. The scheme-theoretic support of $\mathscr{R}_y$ is a finite scheme $\Spec B$ over $\Spec K$. Thus, we see that
	\begin{align*}
		\mathscr{R}_y\cong i_*\ywidetilde M,
	\end{align*} 
	where $i:\Spec B \to X_y$ and $M=\Gamma(\Spec B,\ywidetilde M)\cong\Gamma(X_y,\mathscr{R}_y)$. Supposing that $(x,y)$ is in the support (otherwise the fiber is 0), we obtain a map $B\to \Omega$; hence,
	\begin{align*}
		\mathrm H^0(\mathscr{R}_{(x,y)})\cong\Omega\otimes _BM.
	\end{align*} 
	In particular, the restriction of $\mathrm H^0(q_{\mathscr{R}})$ to $(x,y)$ is a map
	\begin{align*}
		\Omega^n\xrightarrow{\Omega\otimes _K\alpha_y}\Omega\otimes _KM\to \Omega\otimes _BM,
	\end{align*} 
	where
	\begin{align*}
		\alpha_y:K^n\xrightarrow{\cong}\Gamma(X_y,\mathscr{R}_y)
	\end{align*} 
	The first arrow in the composition is an isomorphism, and the second arrow is surjective. Indeed, the base change over $B$ is obtained from the base change over $K$ by imposing additional relations. Thus, we see that by the representing property of $\Quot ^{\text{Adh}}_{\mathscr{O}^n_X}(X)$, we have defined a map
	\begin{align*}
		f:Y\to \Quot ^{\text{Adh}}_{\mathscr{O}^n_X}(X).
	\end{align*} 
Since the adjoint of the map $q_{\mathscr{R}}$ is an equivalence
\begin{align*}
	\mathscr{O}^n_{Y}\xrightarrow{\cong}p_{Y*}\mathscr{R},
\end{align*} 
we see that $f$ factors through the open subfunctor $U\coloneqq \Quot ^{\text{Adh,}n,\text{fr}}_{\mathscr{O}^n_X}(X)\subset \Quot ^{\text{Adh}}_{\mathscr{O}^n_X}(X)$. Going through the same procedure, one can define an inverse by using the universal quotient
\begin{align*}
	q_U: \mathscr{O}^n_{X\times U}\to \mathscr{E},
\end{align*} 
where $\mathscr{E}$ is the universal family on $X\times U$, defined by the same procedure as $\mathscr{R}$. The adjoint to $q_U$ is
\begin{align*}
	 \mathscr{O}^n_U\xrightarrow{\eta}p_{U*}p_{U}^*\mathscr{O}^n_U\xrightarrow{p_{U*}q_U}p_{U*}\mathscr{E};
\end{align*} 
by definition of $q_U$, this is an equivalence. Hence, we have a well-defined map
\begin{align*}
	g:U\to Y.
\end{align*} 
It is easy to see that $ f$ and $g$ are mutually inverse.
\end{proof} 
\subsection{Goldman Symplectic Structure}
Throughout this section, $k=\mathbb{C} $. Let $\Sigma$ be a genus 1 oriented, closed surface, so that $A=k[\pi_1(\Sigma)]\cong \mathscr{O}(\mathbb{G} _m^2)$. Let $X=\Spec A$, and equip $X$ with the Calabi-Yau structure determined by
\begin{align*}
	\theta=\frac{dx}{x}\wedge \frac{dy}{y},
\end{align*} 
where $A=k[x^{\pm 1},y^{\pm 1}]$ and $x,y\in\pi_1(\Sigma)$ with intersection number $x\cdot y=1$. We will show that the symplectic structure defined by Theorem \ref{thm:symplectic-structure} on $\mathsf{Coh}^n_0(X)$ induces the Goldman Poisson structure on the classical character variety.
\begin{prop}
	Let $X=\Spec A$ equipped with the Calabi-Yau structure $\theta$. Restriction of the symplectic form $\omega$ defined by Theorem \ref{thm:symplectic-structure} to $\mathsf{Coh}^n_0(X)^{\text{cl}}$ induces a Poisson bracket on the character variety $\Hom(\pi_1(\Sigma),\GL_n)\git \GL_n$, which is Goldman's Poisson bracket for the pairing
	\begin{align*}
		B:\mathfrak{gl}_n\times \mathfrak{gl}_n\to &\mathbb{C} \\
		(U,V)\mapsto &\tr(UV).
	\end{align*} 
\end{prop} 
\begin{proof}
	By \cite[Proposition 5.3]{BD1}, the Hochschild class which is defined by the fundamental class of $\Sigma$, itself defined by the chosen orientation, is given by $u_*[\Sigma]\in \mathrm H_2(\mathscr{L}\Sigma, \mathbb{C} )$, where $u:\Sigma \to \mathscr{L}\Sigma$ is the constant loop inclusion of $\Sigma$ into its free loop space $\mathscr{L}\Sigma\coloneqq \mathsf{Map}(S^1,\Sigma)$. As is well known, the homology of the free loop space is the Hochschild homology of $X$; hence, $u_*[\Sigma]$ is represented by a class in $\mathrm H_2(\mathscr{L}\Sigma,\mathbb{C} )\cong \mathrm H\mathrm H_2(A)\cong \mathrm H\mathrm H_2(\mathsf{QCoh}(X))$. Now, for any group $G$, there is a map from the bar complex of $G$ to Hochschild chains of $\mathbb{C} [G]$ given by
	\begin{align*}
		[g_1,\ldots,g_m]\mapsto (g_1\cdots g_m)^{-1}\otimes g_1\otimes \cdots \otimes g_m,
	\end{align*} 
	and this is the map which models $u_*$. Thus, for $G=\pi_1(\Sigma)$, the class $[x,y]-[y,x]$ is sent to
	\begin{align*}
		(xy)^{-1}\otimes x\otimes y-(xy)^{-1}\otimes y\otimes x.
	\end{align*} 
	Under normalized Hochschild-Kostant-Rosenberg (HKR), which sends
	\begin{align*}
		g_0\otimes g_1\otimes g_2\mapsto \frac{1}{2}g_0dg_1\wedge dg_2,
	\end{align*} 
	this class gets sent to $d\ln x\wedge d\ln y=\theta$. Moreover, by \cite[Lemma 5.10]{BD1}, it has a unique negative cyclic lift as $\mathrm H\mathrm H_i(X)=0$ for $i>2$ (by e.g. HKR). Thus, the Calabi-Yau structure determined by the orientation of $\Sigma$ agrees with the structure determined by $\theta$. Furthermore, by \cite[Theorem 5.5]{BD}, this Calabi-Yau class induces the symplectic form $\omega$. Now, we need only apply \cite[Lemma 3.7]{BD-is-PTVV} restricted to local systems of rank $n$ vector bundles in degree 0 to obtain that $\omega$ is the pullback along the equivalence (\cite[Corollary A.1]{BRY})
	\begin{align*}
		\mathsf{Coh}^n_0(X)\xrightarrow{\cong}\Map(\Sigma_B,B\GL_n)
	\end{align*} 
	of the symplectic form of \cite[Theorem 2.5]{PTVV}, where $\Sigma_B$ is the constant (``Betti'') stack with value $\Sigma$. Finally, this latter symplectic form is known to recover the Goldman bracket \cite[Section 3.1]{PTVV}.
\end{proof}

\newpage
\printbibliography[heading=bibintoc]
\end{document}